\documentclass[12pt,reqno]{amsart}

\usepackage[utf8]{inputenc}

\usepackage{amsmath} %General package for maths (loads also amsbsy to make bold symbols)
\usepackage{amsthm} %Package for theorems
\usepackage{amssymb} %Symbols (loads also amsfonts)
\usepackage{amscd} %Package for rectangular diagrams
\usepackage{enumitem} %Enumerations
\usepackage{mathtools} %More symbols, etc.
\usepackage{mathrsfs} %Calligraphic symbols with \mathscr
\usepackage{esint} %To write averaged integrals

\usepackage{graphicx} %For using \includegraphics
\usepackage{subcaption}
\usepackage{pgf,tikz,pgfplots}
\usetikzlibrary{arrows}
\usetikzlibrary{intersections}
\usetikzlibrary{fadings}

\usepackage[colorlinks]{hyperref} %Makes references hyperlinks
\hypersetup{
	linktocpage = true,
	colorlinks=true,
	linkcolor=black,
	citecolor=blue,
}
\usepackage{cleveref,autonum}
\usepackage[colorinlistoftodos]{todonotes} %Package to add TODO notes in colors
\newtheorem{theorem}{Theorem}[section]
\newtheorem*{theorem*}{Theorem}
\newtheorem{proposition}[theorem]{Proposition}
\newtheorem{lemma}[theorem]{Lemma}
\newtheorem{corollary}[theorem]{Corollary}

\theoremstyle{definition}
\newtheorem{definition}[theorem]{Definition}

\theoremstyle{remark}
\newtheorem{remark}[theorem]{Remark}
\newtheorem*{remark*}{Remark}

\newcommand{\con}[1]{\mathbb{#1}}
\newcommand{\R}{\con{R}} %Real
\newcommand{\N}{\con{N}} %Natural
\newcommand{\ncal}{\mathcal{N}}

\usepackage{euscript}
\usepackage{relsize}
\DeclareMathAlphabet{\mathpzc}{OT1}{pzc}{m}{it}
\DeclareMathAlphabet\euscr{T1}{qzc}{m}{n}

\makeatletter
\newcommand{\leqnomode}{\tagsleft@true\let\veqno\@@leqno}
\newcommand{\reqnomode}{\tagsleft@false\let\veqno\@@eqno}
\makeatother

\newcommand{\norm}[1]{\left \| {#1} \right \| }
\newcommand{\seminorm}[1]{\left [ {#1} \right ] }

\newcommand{\s}{s}
\newcommand{\fraclaplacian}{(-\Delta)^\s}
\newcommand{\halflaplacian}{(-\Delta)^{1/2}}

\newcommand{\loc}{\mathrm{loc}}

\renewcommand{\d}{\,\mathrm{d}} %straight d with small space before
\newcommand{\df}{\mathrm{d} } %straight d without space 
\newcommand{\dx}{\,\mathrm{d}x} %The usual differential of x
\newcommand{\dint}{\int \! \int}
\newcommand{\average}{\fint}

\newcommand\beqc[1]{\left\{\begin{array}{#1}}
	\newcommand\eeqc{\end{array} \right.}
\def\PDEsystem{rcll}
\def\bmatrix{\begin{pmatrix}}
	\def\ematrix{\end{pmatrix}}

\DeclareMathOperator{\dist}{dist}

\newcommand{\usub}{\underline{u}}
\newcommand{\usup}{\overline{u}}

\renewcommand{\leq}{\leqslant}

\renewcommand{\geq}{\geqslant}

\usepackage[textheight=22cm,textwidth=16.6cm,headsep=1cm,centering]{geometry}

\usepackage{setspace}

\numberwithin{equation}{section}

\graphicspath {{Images/}}

\DeclareRobustCommand{\SkipTocEntry}[5]{}

\title[Stable solutions to fractional semilinear equations]{Stable solutions to fractional semilinear equations: uniqueness, classification, and approximation results}

\author{Tomás Sanz-Perela}
\address{T. Sanz-Perela:
	Departament de Matemàtiques i Informàtica, Universitat  de Barcelona, Gran Via de les Corts Catalanes 585, 
	08007 Barcelona, Spain}
\email{tomas.sanz.perela@ub.edu}

\thanks{The author was supported  by grants PID2020-113596GB-I00, PID2021-123903NB-I00, and RED2018-102650-T funded by MCIN/AEI/10.13039/501100011033 and by ``ERDF A way of making Europe'', and by the EPSRC grant EP/S03157X/1}


\begin{document}

\setstretch{1.0740}

%%%%%%%%%%%%%%%%%%%%%%%%%%%%%%%%%%%%%%%%%%%%%%%%%%%%%%%%%%%%%%%%%%%%%%%%%%%%%%%
%%%%%%%%%%%%%%%%%%%%%%%%%%%%%%%%%%%%%%%%%%%%%%%%%%%%%%%%%%%%%%%%%%%%%%%%%%%%%%%
\begin{abstract}
We study stable solutions to fractional semilinear equations $(-\Delta)^s u = f(u)$ in $\Omega \subset \mathbb{R}^n$, for convex nonlinearities $f$, and under the Dirichlet exterior condition $u=g$ in $\mathbb{R}^n \setminus \Omega$ with general $g$.
We establish a uniqueness and a classification result, and we show that weak (energy) stable solutions can be approximated by a sequence of bounded (and hence regular) stable solutions to similar problems.

As an application of our results, we establish the interior regularity of weak (energy) stable solutions to the problem for the half-Laplacian in dimensions $1 \leq n \leq 4$.

\end{abstract}
%%%%%%%%%%%%%%%%%%%%%%%%%%%%%%%%%%%%%%%%%%%%%%%%%%%%%%%%%%%%%%%%%%%%%%%%%%%%%%%
%%%%%%%%%%%%%%%%%%%%%%%%%%%%%%%%%%%%%%%%%%%%%%%%%%%%%%%%%%%%%%%%%%%%%%%%%%%%%%%

\maketitle

%\vspace{1cm}
\setcounter{tocdepth}{1}
\tableofcontents

%%%%%%%%%%%%%%%%%%%%%%%%%%%%%%%%%%%%%%%%%%%%%%%%%%%%%%%%%%%%%%%%%%%%%%%%%%%%
%%%%%%%%%%%%%%%%%%%%%%%%%%%%%%%%%%%%%%%%%%%%%%%%%%%%%%%%%%%%%%%%%%%%%%%%%%%%
\section{Introduction}
%%%%%%%%%%%%%%%%%%%%%%%%%%%%%%%%%%%%%%%%%%%%%%%%%%%%%%%%%%%%%%%%%%%%%%%%%%%%
%%%%%%%%%%%%%%%%%%%%%%%%%%%%%%%%%%%%%%%%%%%%%%%%%%%%%%%%%%%%%%%%%%%%%%%%%%%%

There is a common property among many scenarios in nature: the observed state of a system is that which, in some sense, minimizes an energy (or action).
When the system configuration is described by a function of several variables, this usually gives rise to a PDE.
A prominent example consists of considering, for a given domain $\Omega\subset \R^n$, the functional
\begin{equation}
	E[u] = \dfrac{1}{2} \int_\Omega |\nabla u|^2 \d x - \int_\Omega F(u) \d x.
\end{equation}
Critical points of this functional satisfy the Euler-Lagrange equation $-\Delta u = f(u)$ in $\Omega$, with $f=F'$. 
This type of reaction-diffusion equations has been used through many years as a basic model for  combustion of substances, phase transitions, and population dynamics.

When studying physically observable states of a system, one looks for solutions which are not only critical points of $E[\cdot]$, but which are also \textit{local minimizers}.
More generally, the class of solutions taken into account is the class of \textit{stable solutions}.
These are solutions at which the second variation of $E[\cdot]$ is nonnegative.
This fact provides some extra information which, together with the PDE, in many cases yields the rigidity and/or regularity of stable solutions.
This kind of properties has been investigated in the last decades from some part of the PDE community ---see the monograph \cite{Dupaigne} and also the introduction of \cite{CabreFigalliRosSerra-Dim9}.

The regularity of stable solutions to $-\Delta u = f(u)$ has been a long-standing problem in elliptic PDEs since the 1970s.
After important efforts devoted to investigate the optimal dimension up to which stable solutions are bounded, the problem has been recently solved by Cabré, Figalli, Ros-Oton, and Serra~\cite{CabreFigalliRosSerra-Dim9}, by proving that stable solutions  are regular in dimensions $n\leq 9$  for all nonnegative nonlinearities $f$.
%The result is optimal, since there exist examples of singular $H^1$ stable solutions in dimensions $n \geq 10$. 
Our main motivation for this paper was the study of the same problem in the nonlocal framework, where one replaces $-\Delta$ by the fractional Laplacian ---see \eqref{Eq:DefFracLap} below---, the most canonical example of integro-differential operator used to model diffusion with long-range interactions.
In this case, the few known results\footnote{We refer to the introduction of  \cite{CabreSanzPerela-HalfLaplacian} for a more precise description of the state of the art of the problem, and also to Figure~1 in that paper for a graphical summary of the known results.} (mainly those contained in the four papers~\cite{RosOtonSerra-Extremal, RosOton-Gelfand, SanzPerela-Radial, CabreSanzPerela-HalfLaplacian}) reach the expected optimal dimension for boundedness of stable solutions only when $f(u) = \lambda e^u$, even in the radial case.

The recent paper \cite{CabreSanzPerela-HalfLaplacian} succeeded in extending some of the techniques of \cite{CabreFigalliRosSerra-Dim9} to the fractional setting, obtaining a universal Hölder estimate in dimensions $1\leq n\leq 4$ for semilinear equations driven by the half-Laplacian.
That result (stated with precision in \Cref{Th:Holder} below) is an a priori estimate for regular stable solutions. 
One of the main motivations for this paper was to provide the necessary tools in order to use the result in \cite{CabreSanzPerela-HalfLaplacian} (and possible future estimates) to establish regularity of weak (energy) stable solutions (see \Cref{Def:EnergySolution}).
This is done in this work through an approximation argument. 
We show that weak (energy) stable solutions can be approximated by a sequence of regular stable solutions to similar problems with different nonlinearities.
Then, the key point is to have at hand universal estimates not depending on the nonlinearity, as the one obtained in \cite{CabreSanzPerela-HalfLaplacian}, to carry the estimates to the limit.

One of the purposes of this article is to provide some general results that can be used in the future when studying stable solutions ---at the moment the only known universal estimate for stable solutions in the fractional setting is that of \cite{CabreSanzPerela-HalfLaplacian}.
For this, and although usually one considers nonnegative stable solutions vanishing outside the domain where the equation holds, for the sake of generality in this paper we will consider stable solutions which are bounded by below, and which satisfy a rather general exterior condition.

Before stating our main results in \Cref{Sec:MainResults}, we need to define properly the notions of solution used in this paper, as well as the definition of stable solution.
We also settle some notation.

%%%%%%%%%%%%%%%%%%%%%%%%
\subsection{Definitions and notation}
%%%%%%%%%%%%%%%%%%%%%%%%

In this paper, we study stable solutions (see \Cref{Def:Stable} below) to
\begin{equation}
	\label{Eq:SemilinearProblem}
	\beqc{\PDEsystem}
	\fraclaplacian u & = & f(u) & \text{ in } \Omega,\\
	u & = & g & \text{ in } \R^n\setminus \Omega,
	\eeqc
\end{equation}
where $\Omega$ is a smooth bounded domain of $\R^n$ and $\fraclaplacian$ is the fractional Laplacian,
\begin{equation}
	\label{Eq:DefFracLap}
	\fraclaplacian w\,(x) := c_{n,\s} \int_{\R^{n}} \dfrac{w(x) - w(y)}{|x-y|^{n + 2\s}}  \d y ,\quad \s \in (0,1),
\end{equation}
with  $c_{n,\s}$ being a positive normalizing constant.
Through the article, $f$ will be a convex nonlinearity, and we will consider general exterior datum $g$ being measureable and satisfying 
\begin{equation}
	\label{Eq:Assumptionsg}
	|g(x) - g(z) | \leq C_0 |x-z|^{\alpha_0}  \quad  \text{ for all } x \in \R^n \setminus \Omega \text{ and } \ z\in \partial \Omega
\end{equation}
for some $\alpha_0 \in (\max\{0, 2\s - 1\},  \s)$ and some positive constant $C_0$.
This assumption provides the regularity near $\partial \Omega$ and the growth at infinity that guarantees that the three notions of solution considered through the article are well defined for such exterior conditions $g$.
For a more detailed discussion on how much this condition can be relaxed depending on the class of solution considered, see \Cref{Sec:ExteriorCondition}.

We define $d_\Omega := \dist(\cdot, \partial \Omega)$, which denotes the distance to $\partial \Omega$, and 
\begin{equation}
	\ncal_\s w (x) := c_{n,\s} \int_\Omega \dfrac{w(x) - w(y)}{|x-y|^{n + 2\s}} \d y \quad \text{ for } x \in \R^n\setminus \Omega,
\end{equation}
which is usually called \textit{nonlocal normal derivative} for its analogies with the classical normal derivative (in terms of the Neumann problem and the integration by parts formula for the fractional Laplacian, see \cite{DipierroRosOtonValdinoci-Neumann}).\footnote{
Note that if $w \equiv 0$ in $\R^n \setminus \Omega$, then  $\ncal_\s w = \fraclaplacian w$ in $\R^n \setminus \Omega$.}

The study of stable solutions usually involves three different notions of solutions to the Dirichlet problem
\begin{equation}
	\label{Eq:DirichletPbLinear}
	\beqc{\PDEsystem}
	\fraclaplacian u & = & h& \text{ in } \Omega,\\
	u & = & g & \text{ in } \R^n\setminus \Omega,
	\eeqc
\end{equation}
which we define properly next.
This first one is the notion of $L^1$\textit{-weak solution}. 
Although in the literature these are called sometimes \textit{very weak solutions} or \textit{distributional solutions}, we prefer to use the same terminology as is \cite{Dupaigne}.\footnote{$L^1$-weak solutions can be equivalently defined through the Green kernel when $g=0$, see~ \cite{AbdellaouiFernandezLeonoriYounes-CalderonZygmund}.}

\begin{definition}
	\label{Def:L1WeakSol}
	Given $\Omega\subset \R^n$ a smooth bounded domain, and given $h \in L^1(\Omega, d_\Omega^\s \dx)$ and $g$ satisfying \eqref{Eq:Assumptionsg}, we say that $u$ is an $L^1$-\textit{weak solution} to \eqref{Eq:DirichletPbLinear} if $u \in L^1(\Omega)$, $u = g$ in $\R^n\setminus \Omega$, and
	\begin{equation}
		\int_{\Omega} u \fraclaplacian \varphi \d x + \int_{\R^n \setminus \Omega} g \ncal_\s \varphi \d x = \int_{\Omega} h \varphi \d x
	\end{equation} 
	for every $\varphi$ such that $\varphi$ and $\fraclaplacian \varphi$ are bounded in $\Omega$ and such that $\varphi \equiv 0$ in $\R^n\setminus \Omega$.\footnote{Note that by standard regularity for the fractional Laplacian, the test functions $\varphi$ that we consider are $C^\s(\overline{\Omega})$.
		In particular, for all $x\in \Omega$ it holds $|\varphi(x)| \leq C d_\Omega^\s (x)$ for some constant $C$ ---this is why the integrability assumption on $h$ involves the weight $d_\Omega^\s$.}
\end{definition}

The second notion of solution that we will consider is the one naturally associated with the variational formulation of the problem, and we give it next.  
It requires the solution to belong to the energy space $H^\s_\Omega := \{ w:\R^n \to \R \, : \, w \in L^2(\Omega) \text{ and } [w]^2_{H^\s_\Omega} < +\infty\}$, where
\begin{equation}
	 [w]^2_{H^\s_\Omega} := \dfrac{c_{n,\s}}{2} \dint_{\R^{2n}\setminus (\Omega^c \times \Omega^c)} \dfrac{| w(x) - w(y) |^2}{|x-y|^{n + 2\s}} \d x \d y,
\end{equation}
and $\Omega^c : = \R^n\setminus \Omega$.
Note that $H^\s_\Omega$ is a separable Hilbert space with norm $
	\norm{\cdot}^2_{H^\s_\Omega} := \norm{\cdot}_{L^2(\Omega)}^2 + [\cdot]^2_{H^\s_\Omega},
$
and scalar product given by $\langle \cdot, \cdot \rangle_{L^2(\Omega)} + \langle \cdot, \cdot \rangle_{H^\s_\Omega}$, where $\langle \cdot, \cdot \rangle_{L^2(\Omega)}$ is the standard scalar product in $L^2(\Omega)$ and
\begin{equation}
	\langle v, w \rangle_{H^\s_\Omega} := \dfrac{c_{n,\s}}{2}\dint_{\R^{2n}\setminus (\Omega^c \times \Omega^c)} \dfrac{\big ( v(x) - v(y) \big)\big ( w(x) - w(y) \big)}{|x-y|^{n + 2\s}} \d x \d y.
\end{equation}
Through the paper we will also use the space $H^\s_{\Omega,0} := \{ w \in  H^\s_\Omega \ : \ w = 0 \ \text{ a.e. in } \R^n \setminus \Omega\}$, which is the closure of $C^\infty_c(\Omega)$ under the norm $\norm{\cdot}_{H^\s_\Omega}$ (at least for $\Omega$ smooth enough).
Note that thanks to the fractional Sobolev inequality, if $\Omega$ is a bounded smooth domain then the fractional Poincaré inequality holds and therefore $\norm{\cdot}_{H^\s_\Omega}$ and $[\cdot]_{H^\s_\Omega}$ are equivalent norms in $H^\s_{\Omega,0}$ ---and, as a byproduct, $\langle \cdot, \cdot \rangle_{H^\s_\Omega} $ is a scalar product in $H^\s_{\Omega,0}$.
For more details, see \cite[Section~3.1]{CozziPassalacqua} and the references therein.

\begin{definition}
	\label{Def:EnergySolution}
	Given $\Omega\subset \R^n$ a smooth bounded domain, $h\in L^2(\Omega)$, and  $g$ satisfying \eqref{Eq:Assumptionsg}, we say that $u$ is an \textit{energy solution} (or \textit{variational solution}) to \eqref{Eq:DirichletPbLinear} if $u\in H^\s_\Omega$, $u = g$ in $\R^n\setminus\Omega$, and
	\begin{equation}
		\langle u, \varphi \rangle_{H^\s_\Omega} = \dfrac{c_{n,\s}}{2}\dint_{\R^{2n}\setminus (\Omega^c \times \Omega^c)} \dfrac{\big ( u(x) - u(y) \big)\big ( \varphi(x) - \varphi(y) \big)}{|x-y|^{n + 2\s}} \d x \d y= \int_{\Omega} h \varphi \d x.
	\end{equation}
	for all $\varphi\in H^\s_\Omega$ with $\varphi \equiv 0$ in $\R^n \setminus \Omega$.
\end{definition}

Note that the test functions considered in the energy formulation belong to $H^\s(\R^n)$, and that by density it is enough to consider test functions $\varphi \in C^\infty_c(\Omega)$.
Using integration by parts (since $\varphi$ is regular enough), one sees that
\begin{equation}
	\begin{split}
		\langle u, \varphi \rangle_{H^\s_\Omega} 
		&= \int_\Omega  u \, \fraclaplacian \varphi \d x + \int_{\R^n \setminus \Omega} u \, \ncal_\s \varphi \d x ,
	\end{split}
\end{equation}
showing that energy solutions are always $L^1$-weak solutions.

\begin{remark}
	\label{Remark:L1WeakSolWhichAreHs}
	If $u$ is an $L^1$-weak solution to \eqref{Eq:DirichletPbLinear} and $u\in H^\s_{\Omega}$, by taking $\varphi\in C^\infty_c(\Omega)$ in the weak formulation we can integrate by parts to obtain that
	\begin{equation}
		\label{Eq:L1WeakFormulationHs}
		\langle u, \varphi \rangle_{H^\s_\Omega} = \int_\Omega h \varphi \d x.
	\end{equation}
	Then, by density this holds for every $\varphi \in H^\s_{\Omega,0}$ such that $h \varphi \in L^1(\Omega)$.
	We will consider this setting in most of the results of this paper, exploiting the previous variational-like formulation.
	As a terminological remark, we note that since the right-hand sides that we consider through the paper are merely in $L^1(\Omega)$, we will not call our solutions `energy solutions'.
\end{remark}

The third notion of solution considered in this article is the classical one.

\begin{definition}
	\label{Def:PointwiseSolution}
	Given $\Omega\subset \R^n$ a smooth bounded domain, $h\in C(\Omega)$, and  $g$ satisfying \eqref{Eq:Assumptionsg}, we say that $u$ is a \textit{pointwise solution} (or \textit{classical solution}) to \eqref{Eq:DirichletPbLinear} if $u\in C(\overline{\Omega})$ and \eqref{Eq:DirichletPbLinear} is satisfied pointwise.\footnote{A typical assumption to ensure that $\fraclaplacian u$ is well defined pointwise is that $u\in C^2(\Omega)\cap L^1_\s(\R^n)$ (see \eqref{Eq:DefL1s} for this last space). The $C^2$ assumption can be relaxed depending on $\s\in (0,1)$, requiring only $u\in C^\alpha(\Omega)$ for some $\alpha > 2\s$. The required integrability at infinity is satisfied  under assumption \eqref{Eq:Assumptionsg}.}
\end{definition}

\begin{remark}
	\label{Remark:BoundedSolutions}
	For the semilinear problem \eqref{Eq:SemilinearProblem} the three notions of solution coincide when $u$ is bounded in $\Omega$. 
	Indeed, if $u$ is an $L^1$-weak solution to \eqref{Eq:SemilinearProblem} which is bounded in $\Omega$, then it is a classical solution (and by assumption \eqref{Eq:Assumptionsg} also an energy solution, arguing as in ~\Cref{Sec:ExteriorCondition}).
	To see this, one first uses that the fractional Laplacian commutes with convolution and, by considering a standard mollifier $\eta_\varepsilon$, it follows that	$\fraclaplacian(u * \eta_\varepsilon )  = f(u)* \eta_\varepsilon$ in any smaller domain $\Omega' \subset \subset \Omega$ for $\varepsilon$ small enough depending on $\Omega'$.
	Using the usual interior estimates for the fractional Laplacian (see \cite[Corollaries~2.4 and 2.5]{RosOtonSerra-Regularity}) and letting $\varepsilon\to 0$ it follows that $u\in C^\alpha(\Omega)$ for some $\alpha > 2 \s$ and thus $\fraclaplacian u = f(u)$ holds pointwise in $\Omega$.
	Moreover, by the results of Audrito and Ros-Oton \cite{AudritoRosOton}, $u$ is continous up to the boundary.\footnote{The results of \cite{AudritoRosOton} hold for $L^1$-weak solutions since are based on the comparison principle; see \cite[Remark~1.5]{AudritoRosOton}.}
\end{remark}

We recall now the definition of stability in the fractional setting.

\begin{definition}
	\label{Def:Stable}
	Let $u$ be a solution (in any of the previous three senses defined above) to the semilinear problem \eqref{Eq:SemilinearProblem} with $f\in C^1(\R)$.
	We say that $u$ is \textit{stable} in $\Omega$ if $f'(u) \in L^1_\loc(\Omega)$ and
	\begin{equation}
		\label{Eq:Stability}
		\int_\Omega f'(u) \xi^2 \d x \leq [\xi]^2_{H^\s_{\Omega}} \quad \text{ for every } \xi \in C^\infty_c(\Omega).
	\end{equation}
	
\end{definition}

Note that if one considers the energy functional associated to problem \eqref{Eq:SemilinearProblem},
\begin{equation}
	E[w] = \dfrac{1}{2}\seminorm{w}_{H^{\s}_\Omega}^2 -  \int_\Omega F(w) \d x \quad \text{ with } F' = f,
\end{equation}
then the stability condition \eqref{Eq:Stability} is equivalent to
\begin{equation}
	\dfrac{\df^2}{\df \varepsilon^2} E[u + \varepsilon \xi] \geq 0 \quad \text{ for every } \xi \in C^\infty_c(\Omega)
\end{equation}

%%%%%%%%%%%%%%%%%%%%%%%%%%%%%
\subsection{Main results}
%%%%%%%%%%%%%%%%%%%%%%%%%%%%%
\label{Sec:MainResults}

Let us now present the main results of this article.
We will always assume that the nonlinearity $f$ is strictly convex, although after the proof of each result we will make some comments on the (possible) differences if one allows $f$ to be affine.
 
Our first result is the following uniqueness theorem for stable solutions.
It is proved in \Cref{Sec:Uniqueness}.

\begin{theorem}
	\label{Th:UniquenessStableNonlinear}
	Let $\Omega\subset \R^n$ a bounded smooth domain of $\R^n$ and let $f\in C^1(\R)$ be a strictly convex function.
	
	Then, there exists at most one stable solution $u$ to \eqref{Eq:SemilinearProblem} ---in the $L^1$-weak sense and for $g$ as in \eqref{Eq:Assumptionsg}--- such that $u$ is bounded by below and $u\in H^\s_{\Omega}$.
\end{theorem}

Our next result states that when $u$ vanishes in $\R^n \setminus \Omega$ and $f(0)=0$, then $u\equiv0$ is the only possible nonnegative stable solution.
We prove it in \Cref{Sec:Classification}.

\begin{theorem}
	\label{Th:Classiff(0)=0}
	Let $\Omega\subset \R^n$ be a bounded smooth domain of $\R^n$ and let $f\in C^1(\R)$ be a strictly convex function with $f(0)=0$.
	Let $u\in H^\s_{\Omega}$ be a nonnegative $L^1$-weak stable solution to 
	\begin{equation}
		\label{Eq:SemilinearProblemZero}
		\beqc{\PDEsystem}
		\fraclaplacian u & = & f(u) & \text{ in } \Omega,\\
		u & = & 0 & \text{ in } \R^n\setminus \Omega.
		\eeqc
	\end{equation}
	
	Then, $u \equiv 0$ and $f'(0) \leq \lambda_1$, where $\lambda_1 >0$ is the first eigenvalue of $\fraclaplacian$ in $\Omega$ with zero exterior Dirichlet condition.	
\end{theorem}

Finally, we present our approximation result for stable solutions.
It is proved in \Cref{Sec:Approximation}.

\begin{theorem}
	\label{Th:HsAppConvex}
	Let $\Omega\subset \R^n$ be a bounded smooth domain, and let $f\in C^1(\R)$ be a strictly convex function.
	Let $u$ be a stable $L^1$-weak solution to \eqref{Eq:SemilinearProblem} for some $g$ as in \eqref{Eq:Assumptionsg}.
	Assume that $u\in H^\s_\Omega$, that $u\geq -M$ in $\R^n$ for some $M$, and that $f(-M) \geq 0$. 
	
	Then, there exists a nondecreasing sequence $\{f_k\}_{k \in \N} \subset C^1([-M,+\infty))$ of globally Lipschitz convex nonlinearities converging pointwise to $f$ in $[-M,+\infty)$, and a sequence $\{u_k\}_{k \in \N}$ of bounded stable solutions to
	\begin{equation}
		\label{Eq:SemilinearPbHsAppfkConvex}
		\beqc{\PDEsystem}
		\fraclaplacian u_k & = & f_k(u_k)& \text{ in } \Omega,\\
		u_k & = & g& \text{ in } \R^n\setminus \Omega,
		\eeqc
	\end{equation}
	such that $-M \leq u_k \leq u$ in $\R^n$ and such that $u_k \to u$ in $H^\s_\Omega$  as $k\to +\infty$.
	
	Moreover, if $f$ is nonnegative, then all $f_k$ are nonnegative as well, and if $f\in C^{1,\gamma}(\R)$ for some $\gamma \in (0,1)$, then $f_k\in C^{1,\gamma}(\R)$ for all $k\in \N$.
\end{theorem}

As mentioned before, the previous result can be combined with uniform a priori estimates for stable solutions to establish the regularity of $H^\s_\Omega$ stable solutions. 
In view of the main result of \cite{CabreSanzPerela-HalfLaplacian} (see \Cref{Th:Holder} below), our approximation theorem yields the following regularity result in low dimensions.

\begin{corollary}
	\label{Th:HolderWeak}
	Let $\Omega\subset \R^n$ be a bounded smooth domain, and let $f$ be a nonnegative strictly convex $C^{1,\gamma}$ function for some $\gamma >0$. 
	Let $u$ be a stable $L^1$-weak solution to \eqref{Eq:SemilinearProblem} for $\s=1/2$ and for some $g$ as in \eqref{Eq:Assumptionsg}.
	Assume that $u\in H^{1/2}_\Omega$ and that $u$ is bounded by below in $\R^n$.
	
	Then, $u\in C^{2,\delta}(\Omega)$ for some $\delta>0$ provided that $1 \leq n\leq 4$.
\end{corollary}

\subsection{Organization of the paper:}
\Cref{Th:UniquenessStableNonlinear,Th:Classiff(0)=0} are proved respectively in \Cref{Sec:Uniqueness} and \Cref{Sec:Classification}.
Our main approximation result, \Cref{Th:HsAppConvex}, is established in \Cref{Sec:Approximation}, where we also prove another approximation theorem for stable solutions (\Cref{Th:HsAppEps}).
\Cref{Th:HolderWeak} is proved in \Cref{Sec:Applications} and, finally, in \Cref{Sec:Counterexamples} we provide a counterexample for our approximation result when we drop the assumption of $u$ belonging to $H^\s_\Omega$.

The article contains three appendices.
In \Cref{Sec:IVTh} we recall an ``intermediate value theorem'' for functions in fractional Sobolev spaces which is needed in the arguments of \Cref{Remark:StrictConvexity}.
\Cref{Sec:ExteriorCondition} is devoted to some comments on the regularity assumptions on the exterior condition $g$ needed to define each notion of solution.
Finally, in \Cref{Sec:L1Theory} we collect several results concerning $L^1$-weak solutions which are used through the paper.

%%%%%%%%%%%%%%%%%%%%%%%%%%%%%%%%%%%%%%%%%%%%%%%%%%%%
\section{Uniqueness of stable solutions}
%%%%%%%%%%%%%%%%%%%%%%%%%%%%%%%%%%%%%%%%%%%%%%%%%%%%
\label{Sec:Uniqueness}

In this section, we present several results concerning stable solutions to the semilinear problem \eqref{Eq:SemilinearProblem}.
We will conclude by proving our uniqueness result, \Cref{Th:UniquenessStableNonlinear}.

Through the paper we will consider $L^1$-weak solutions which are in $H^\s_{\Omega}$ (exploiting the variational formulation of the equation, see \Cref{Remark:L1WeakSolWhichAreHs}).
Furthermore, we will always assume that $f$ is convex.
Regarding the stability inequality \eqref{Eq:Stability}, through the section we intend to take test functions $\xi \in H^\s_{\Omega,0}$ ---not necessarily belonging to $C^\infty_c(\Omega)$.
As explained in the following remark, \eqref{Eq:Stability} holds as well for this wider class of test functions thanks to the convexity of~$f$.
Here and through the paper we denote positive and negative parts by $w^+ := \max \{w, 0\}$ and $w^- := \max \{-w, 0\}$ (and thus $w = w^+ - w^-$).

\begin{remark}
	\label{Remark:StabilityHs}
	Let us show that we can take $\xi \in H^\s_{\Omega,0}$ as a test function in the stability condition~\eqref{Eq:Stability}, at least when $f$ is convex and the stable solution $u$ is bounded below in $\Omega$.  
	To do it, it suffices to show that
	\begin{equation}
		\label{Eq:f'xi2Integrable}
		f'(u) \xi^2 \in L^1(\Omega) \quad \text{ for all } \xi \in H^\s_{\Omega,0},
	\end{equation} 
	and use an approximation argument.
	Note first that since $f$ is convex and $u\geq -M$ in $\Omega$ for some $M$, we have $f'(-M) \leq f'(u)$ in $\Omega$, which shows that $f'(u)^- \in L^\infty(\Omega)$, and therefore $f'(u)^- \xi^2 \in L^1(\Omega) $ for every $ \xi \in H^\s_{\Omega,0}$.
	Now, from the stability of $u$ we have
	\begin{equation}
		\label{Eq:StabilityfMinus}
		\int_\Omega f'(u)^+ \xi^2 \d x \leq [\xi]^2_{H^\s_{\Omega}} + \int_\Omega f'(u)^- \xi^2 \d x \quad \text{ for every } \xi \in C^\infty_c(\Omega).
	\end{equation}
	Then, for every $\xi \in H^\s_{\Omega,0}$, let $\{\xi_k\}_{k\in \N}\subset C^\infty_c(\Omega)$ be a sequence converging to $\xi$ in $H^\s_{\Omega,0}$ as $k \to +\infty$.
	Taking $\xi = \xi_k$ in \eqref{Eq:StabilityfMinus}, and using Fatou's lemma, by letting $k\to +\infty$ we obtain that \eqref{Eq:StabilityfMinus} holds for every $\xi \in H^\s_{\Omega,0}$.
	This yields that $f'(u)^+ \xi^2 \in L^1(\Omega)$ for every $\xi \in H^\s_{\Omega,0}$,	establishing \eqref{Eq:f'xi2Integrable} and, as a byproduct, showing that the stability condition \eqref{Eq:Stability} holds for all $\xi \in H^\s_{\Omega,0}$.
\end{remark}

Our first result in this section states that if the nonlinearity $f$ is convex, then a bounded by below stable solution cannot cross another solution which is also bounded by below.

\begin{proposition}
	\label{Prop:OrderStableSolutions}
	Let $\Omega\subset \R^n$ be a bounded smooth domain and let $f\in C^1(\R)$ be a convex function.
	Let $u,v\in H^\s_\Omega$ be two $L^1$-weak  solutions to the problem \eqref{Eq:SemilinearProblem}, for some exterior condition $g$ as in \eqref{Eq:Assumptionsg}.
	Assume that both $u$ and $v$ are bounded by below in $\Omega$ and that at least one solution is stable in $\Omega$.
	
	Then, $u$ and $v$ are ordered: either $u<v$ in $\Omega$, $u\equiv v$ in $\Omega$, or $u>v$ in $\Omega$.
\end{proposition}

\begin{proof}
	Without loss of generality we will assume that $u$ is stable.
	We will also use through the proof that $u\geq -M$ and $v\geq -M$ in $\Omega$ for some $M$.
	
	First, using that for every two numbers $a,b\in \R$, it holds 
	\begin{equation}
		\label{Eq:PositivePartHs}
		|a^+ - b^+|^2 \leq (a - b)(a^+ - b^+) \leq |a-b|^2,
	\end{equation}
	it follows that $(u-v)^+\in H^\s_\Omega$.
	Moreover, since both $u$ and $v$ agree outside $\Omega$,  $(u-v)^+\in H^\s_{\Omega,0}$.
	
	Now, we claim that
	\begin{equation}
		\label{Eq:fu-fvL1}
		\big( f(u) - f(v) \big) (u - v)^+ \in L^1(\Omega).
	\end{equation}
	To show this, note that from the convexity of $f$ and the lower bound of $v$ we have
	\begin{equation}
		f'(-M) |(u-v)^+|^2 \leq f'(v) (u-v)(u-v)^+ \leq \big( f(u) - f(v) \big) (u-v)^+
	\end{equation}
	and
	\begin{equation}
		\big( f(u) - f(v) \big) (u-v)^+ \leq   f'(u) (u-v)(u-v)^+ \leq f'(u) |(u-v)^+|^2.
	\end{equation}
	Using that $(u-v)^+\in H^\s_{\Omega,0}$ and \eqref{Eq:f'xi2Integrable}, the claim follows.

	Note that $u-v\in H^\s_{\Omega,0}$ is an $L^1$-weak solution to
	\begin{equation}
		\beqc{\PDEsystem}
		\fraclaplacian (u-v) & = &  f(u) - f(v) & \text{ in } \Omega, \\
		u - v & = & 0& \text{ in } \R^n\setminus \Omega.
		\eeqc
	\end{equation}
	Testing the above equation with $(u-v)^+\in H^\s_{\Omega,0}$ (thanks to \eqref{Eq:fu-fvL1} and taking into account \Cref{Remark:L1WeakSolWhichAreHs}), and using \eqref{Eq:PositivePartHs}, we get
	\begin{equation}
		[(u-v)^+]^2_{H^\s_\Omega} \leq \langle u - v, (u-v)^+ \rangle_{H^\s_\Omega} = \int_{\Omega} \big( f(u) - f(v) \big) (u-v)^+ \d x .
	\end{equation}
	Combining this inequality with the stability of $u$ (taking $\xi = (u-v)^+\in H^\s_{\Omega,0}$ as a test function in \eqref{Eq:Stability}, see \Cref{Remark:StabilityHs}), we obtain
	\begin{equation}
		\begin{split}
			\int_\Omega f'(u) |(u-v)^+|^2 \d x &\leq  [(u-v)^+]^2_{H^\s_\Omega}  \\
			&\leq \int_{\Omega} \big( f(u) - f(v) \big) (u-v)^+ \d x \\
			& \leq \int_{\Omega} f'(u) |(u-v)^+|^2  \d x,
		\end{split}	
	\end{equation}
	where in this last inequality we have used the convexity of $f$.
	As a consequence, we have that
	\begin{equation}
		[(u-v)^+]^2_{H^\s_\Omega} - \int_{\Omega} f'(u) |(u-v)^+|^2  \d x = 0.
	\end{equation}

	Since $u$ is stable, the functional 
	\begin{equation}
		w \mapsto [w]^2_{H^\s_\Omega} - \int_{\Omega} f'(u) w^2  \d x 
	\end{equation}
	is nonnegative in $H^\s_{\Omega,0}$, and therefore $(u-v)^+$ is a minimizer.
	Hence, it is an energy solution (and in particular $L^1$-weak solution) to
	\begin{equation}
		\beqc{\PDEsystem}
		\fraclaplacian (u-v)^+ &=& f'(u) (u-v)^+ &\text{ in } \Omega,\\
		(u-v)^+ & = & 0 & \text{ in } \R^n\setminus \Omega.
		\eeqc
	\end{equation}
%	Recall that \Cref{Remark:StabilityHs} ---in particular \eqref{Eq:f'xi2Integrable}--- guarantees that the righ-hand side of this equation can be tested against functions in $H^\s_{\Omega,0}$.
	
	Writing $f'(u) = f'(u)^+ - f'(u)^-$, we have (in the $L^1$-weak sense)
	\begin{equation}
		\fraclaplacian (u-v)^+ + f'(u)^- (u-v)^+ = f'(u)^+ (u-v)^+ \geq 0 \quad \text{ in } \Omega.
	\end{equation}
	Then, by the strong maximum principle (we use \Cref{Prop:StrongMaxPple} taking into account that $f'(u)^- \in L^\infty(\Omega)$, since $f'(t) \geq f'(-M)$ for all $t\geq -M$ by convexity) we deduce that either $u > v$ in $\Omega$, or $u \leq v$ in $\Omega$. 
	In this second case, we have (in the $L^1$-weak sense again)
	\begin{equation}
		\fraclaplacian (v-u) = f(v) - f(u) \geq f'(u) (v - u) \quad \text{ in } \Omega,
	\end{equation}
	and as before, using now that $v\geq u$ in $\Omega$, we get
	\begin{equation}
		\fraclaplacian (v-u) + f'(u)^- (v-u) = f'(u)^+ (v-u) \geq 0 \quad \text{ in } \Omega.
	\end{equation}
	The strong maximum principle yields that either $u \equiv v$  in $\Omega$, or $u < v$  in $\Omega$.	
\end{proof}

Next, we show that if $f$ is strictly convex then stable solutions are minimal in the sense that  below a stable solution there cannot exist another solution bounded by below.
After the proof, in \Cref{Remark:StrictConvexity} below, we make some comments on the strict convexity assumption.

\begin{proposition}
	\label{Prop:UniquenessIfNotEig}
	Let $\Omega\subset \R^n$ a bounded smooth domain of $\R^n$ and let $f\in C^1(\R)$ be a strictly convex function.
	Let $u, v\in H^\s_\Omega$ be two $L^1$-weak solutions to \eqref{Eq:SemilinearProblem} for some exterior condition $g$ as in \eqref{Eq:Assumptionsg}.
	Assume that $-M \leq v \leq u$ in $\Omega$ for some $M$ and that $u$ is stable.
	
	Then, $u \equiv v$.
\end{proposition}

\begin{proof}
	By contradiction, we assume that $u \not \equiv v$ in $\Omega$.
	Then, we define $w := u - v\in H^\s_{\Omega,0}$, which is an $L^1$-weak solution to
	\begin{equation}
		\beqc{\PDEsystem}
		\fraclaplacian w &=  &f(u) - f(v)& \text{ in } \Omega,\\
		w & = & 0 & \text{ in } \R^n\setminus \Omega.
		\eeqc
	\end{equation}
	Note that by the strong maximum principle (\Cref{Prop:OrderStableSolutions}), $w > 0$ in $\Omega$.
	Taking $w$ as a test function in the weak formulation of the above problem (arguing as in the proof of \Cref{Prop:OrderStableSolutions} and taking into account \Cref{Remark:L1WeakSolWhichAreHs}), we obtain
	\begin{equation}
		[w]^2_{H^\s_{\Omega}} = \langle w, w \rangle_{H^\s_\Omega} = \int_{\Omega} \big( f(u) - f(v) \big) w \d x.
	\end{equation}
	Using the convexity of $f$ and the fact that  $u$ is stable ---note that we can take $\xi = w$ as a test function in the stability condition, see \Cref{Remark:StabilityHs}---, we get
	\begin{equation}
		[w]^2_{H^\s_\Omega}  \leq  \int_{\Omega} \big( f(u) - f(v) \big) w \d x 	\leq \int_{\Omega} f'(u) w^2 \d x \leq 	[w]^2_{H^\s_\Omega}.
	\end{equation}
	This yields that 
	\begin{equation}
		\label{Eq:ProofStrictConvexityContrad}
		f(u) - f(v) = f'(u)(u - v) \quad \textrm{ in } \Omega,
	\end{equation}
	which is a contradiction with the strict convexity of $f$ (since $u\not\equiv v$).
\end{proof}

\begin{remark}
	\label{Remark:StrictConvexity}
	If we relax the strict convexity assumption in the previous lemma to allow non strict convex nonlinearities, the consequence of the previous result may not be true.
	Indeed, if $f(u) = a + \lambda_1 u$ with $a\in \R$ and $\lambda_1$ being the first Dirichlet eigenvalue of $\fraclaplacian$ in $\Omega$, then the function $w$ defined in the previous proof could be (a positive multiple of) the first Dirichlet eigenfunction of $\fraclaplacian$ in $\Omega$.
	
	Let us show now that, at least if $\s \geq 1/2$, this eigenvalue case is the only obstruction when relaxing the strictness in the convexity assumption.
	More precisely, assuming $f$ (non strictly) convex and $w := u - v \not \equiv 0$, we will prove that \eqref{Eq:ProofStrictConvexityContrad} does not yield a contradiction, but it entails that $f(t) = a + \lambda_1 t$ for all $t \in (\inf_\Omega v, \sup_\Omega u)$, for some $a\in \R$ and with $\lambda_1$ being the first Dirichlet eigenvalue of $\fraclaplacian$ in $\Omega$.
	This will lead to $u-v$ being any (positive multiple of the) first Dirichlet eigenfunction of $\fraclaplacian$ in $\Omega$.
	
%	Let us show the previous statement. 
	To start with, we claim that \eqref{Eq:ProofStrictConvexityContrad} yields that $f$ is affine in the interval $(I,S)$, where $I := \inf_\Omega v$ and $S := \sup_\Omega u$.
	To prove this, we take a sequence of smooth domains $\Omega_k \subset \subset \Omega$ with $\Omega_k \subset \Omega_{k+1}$ and $\cup_{k \geq 1} \Omega_k = \Omega$ such that, for every $k\geq 1$, $\dist (\Omega_k, \partial \Omega) \geq 1/k$ (without loss of generality, after scaling we may assume that $\sup_{x \in \Omega} \dist (x,\partial\Omega) > 1$).
	Then, defining $I_k := \inf_{\Omega_k} v$ and $S_k := \sup_{\Omega_k} u$, it suffices to show that $f$ is affine in $(I_k, S_k)$, and by letting $k\to +\infty$ we will conclude the claim. The details go as follows.

	First, since $f(u) - f(v) = f'(u)(u - v)$ in $\Omega_k$, we have that $f$ is affine in the interval $[v(x), u(x)]$ for a.e. $x\in \Omega_k$.
	Now, by the strong maximum principle (applied to $u-v$ in $\Omega_k$) there exists a positive constant $\delta_k> 0$ such that
	\begin{equation}
		\label{Eq:DistDeltak}
		u - v \geq \delta_k >0 \quad \text{ in } \Omega_k.
	\end{equation}
	As a consequence of this, each of the intervals of the form $[v(x), u(x)]$ for a.e. $x\in \Omega_k$ has length at least $\delta_k$.
	That $f$ is affine in  $(I_k, S_k)$ will follow if we show that the union of the intervals $[v(x), u(x)]$, as $x$ varies a.e. in $\Omega_k$, covers all the interval $(I_k, S_k)$.
	Assume by contradiction that this last assertion is not true. 
	Then, there exist two numbers $a<b$ with $b-a \geq \delta_k$ such that $|\{a \leq u \leq b\} \cap \Omega_k |= 0$ but $|\{ u < a\} \cap \Omega_k |> 0$ and $|\{ u > b\} \cap \Omega_k |> 0$, contradicting the intermediate value theorem for functions in $H^\s$ with  $\s\geq 1/2$; see \Cref{Prop:IVT}.
	
	Now that we have proved our claim, we have that $f(t) = a + bt$ for $t\in (I,S)$ and for some $a,b\in \R$.
	Thus, $w$ solves
	\begin{equation}
		\beqc{\PDEsystem}
		\fraclaplacian w & = & b w & \text{ in } \Omega,\\
		w & > & 0 & \text{ in } \Omega,\\
		w & =& 0 & \text{ in } \R^n\setminus \Omega,
		\eeqc
	\end{equation}
	in the variational sense.
	As a consequence, since $w$ is positive, it follows that $w$ is (a multiple of) the first Dirichlet eigenfunction of $\fraclaplacian$ in $\Omega$ (see \cite[Proposition~9]{ServadeiValdinoci}) and thus necessarily $b = \lambda_1$, the first Dirichlet eigenvalue of $\fraclaplacian$.
\end{remark}

The previous two results yield our main uniqueness result for bounded by below stable solutions in $H^\s_{\Omega}$.

\begin{proof}[Proof of \Cref{Th:UniquenessStableNonlinear}]
	Assume by contradiction that there exist two different stable solutions $u_1, u_2 \in H^\s_{\Omega}$ to \eqref{Eq:SemilinearProblem} (in the $L^1$-weak sense) which are bounded by below.
	Then, by \Cref{Prop:OrderStableSolutions} they are ordered and we may assume that $u_1 < u_2$ in $\Omega$.
	However, since $f$ is strictly convex,  this contradicts \Cref{Prop:UniquenessIfNotEig}.
\end{proof}

%%%%%%%%%%%%%%%%%%%%%%%%%%%%%%%%%%%%%%%%%%%%%%%%%%%%
\section{A classification result}
%%%%%%%%%%%%%%%%%%%%%%%%%%%%%%%%%%%%%%%%%%%%%%%%%%%%
\label{Sec:Classification}

In this section we show \Cref{Th:Classiff(0)=0}, our classification result for stable solutions in $H^\s_{\Omega,0}$ when the nonlinearity $f$ is strictly convex and satisfies $f(0) = 0$.
The proof follows essentially from \Cref{Prop:UniquenessIfNotEig}, but to establish the bound for $f'(0)$ we will use the next lemma.

\begin{lemma}
	\label{Lemma:SuperlinearNonlinearityZeroSolNotStable}
	Let $\Omega\subset \R^n$ be a bounded smooth domain and let $f\in C^1(\R)$ be such that	$f(t) > \lambda_1 t$ for $t>0$, where $\lambda_1$ is the first eigenvalue of $\fraclaplacian$ in $\Omega$ with zero exterior Dirichlet condition.
	Let $u\in H^\s_{\Omega}$ be a nonnegative function such that $\fraclaplacian u \geq f(u)$ in $\Omega$ in the $L^1$-weak sense, i.e., such that $f(u) \in L^1(\Omega, d^\s_{\Omega} \d x)$ and
	\begin{equation}
		\int_{\Omega} u \fraclaplacian \varphi \d x + \int_{\R^n \setminus \Omega} u \ncal_\s \varphi \d x \geq \int_{\Omega} f(u) \varphi \d x
	\end{equation} 
	for every nonnegative function $\varphi$ such that $\varphi$ and $\fraclaplacian \varphi$ are bounded in $\Omega$ and such that $\varphi \equiv 0$ in $\R^n\setminus \Omega$.
	
	Then, $u\equiv 0$ in $\R^n$, $f(0)=0$, and $u$ is a solution to \eqref{Eq:SemilinearProblem} with $g \equiv 0$ which is not stable.
\end{lemma}

\begin{proof}
	From the strong maximum principle (\Cref{Prop:StrongMaxPple}), either $u>0$ in $\Omega$ or $u\equiv0$ in $\R^n$.
	Let us assume by contradiction that $u > 0$ in $\Omega$ and let $w\in H^\s_{\Omega}$ be the (energy) solution to 
	\begin{equation}
		\beqc{\PDEsystem}
		\fraclaplacian w & = & 0 & \text{ in } \Omega,\\
		w & = & u & \text{ in } \R^n \setminus \Omega.
		\eeqc
	\end{equation}
	Since $u-w\in H^\s_{\Omega,0}$, arguing as in \Cref{Remark:L1WeakSolWhichAreHs} using Fatou's lemma (taking into account that $f(u) > 0$ in $\Omega$), we have that
	\begin{equation}
		\label{Eq:SuperlinearNonlinearitySupersol}
		\langle u - w, \varphi \rangle_{H^\s_\Omega} \geq \int_{\Omega} f(u) \varphi \d x
	\end{equation}
	for every nonnegative $\varphi\in H^\s_{\Omega,0}$.

	Now, let $\phi_1\in H^\s_{\Omega,0}$ be the first eigenvalue of $\fraclaplacian$ with zero Dirichlet exterior condition in $\R^n\setminus \Omega$.
	Recall that $\phi_1>0$ in $\Omega$.
	On the one hand, taking $\varphi = \phi_1 $ in \eqref{Eq:SuperlinearNonlinearitySupersol} we have
	\begin{equation}
		\langle u - w, \phi_1 \rangle_{H^\s_\Omega} \geq \int_{\Omega} f(u) \phi_1 \d x > \lambda_1\int_{\Omega} u \phi_1 \d x .
	\end{equation}
	On the other hand, taking $u - w\in H^\s_{\Omega,0}$  as a test function in the weak formulation of the equation for $\phi_1$ and using that $ w\phi_1\geq 0 $ in $\Omega$ (note that $0 \leq w \leq u$ in $\Omega$ by the maximum principle), we have
	\begin{equation}
		\langle  u-w, \phi_1 \rangle_{H^\s_\Omega} = \lambda_1 \int_{\Omega} (u - w) \phi_1 \d x \leq \lambda_1 \int_{\Omega} u \phi_1 \d x,
	\end{equation}
	arriving at a contradiction.
	This shows that $u\equiv0$ in $\R^n$.

	Finally, let us show that $u$ is not stable.
	To do it, recall that by the definition of $\phi_1$ and $\lambda_1$, it holds that
	\begin{equation}
		\lambda_1 = \dfrac{[\phi_1]_{H^\s_{\Omega}}^2}{\norm{\phi_1}^2_{L^2(\Omega)}}.
	\end{equation}
	Hence, since $f'(0) > \lambda_1$, we have
	\begin{equation}
		\int_{\Omega} f'(u) \phi_1^2 \d x = f'(0) \int_{\Omega}  \phi_1^2 \d x > \lambda_1 \norm{\phi_1}^2_{L^2(\Omega)} = [\phi_1]_{H^\s_{\Omega}}^2,
	\end{equation}
	showing that $u$ is not stable.
\end{proof}

With the previous result at hand we can now establish \Cref{Th:Classiff(0)=0}.

\begin{proof}[Proof of \Cref{Th:Classiff(0)=0}]
	Since $f(0)=0$, $v \equiv 0$ is a solution to \eqref{Eq:SemilinearProblemZero}.
	Then, by \Cref{Prop:UniquenessIfNotEig},  $u\equiv v \equiv 0$ and since $u$ is stable, by \Cref{Lemma:SuperlinearNonlinearityZeroSolNotStable} it follows that $f'(0) \leq \lambda_1$.
\end{proof}

\begin{remark}
	\label{Remark:StrictConvexityClassification}
	If we relax the convexity assumption on $f$ to allow non strictly convex nonlinearities, then  \Cref{Th:Classiff(0)=0} may not be true.
	Indeed, if $f(u) = \lambda_1 u$, then $u$ can be the first eigenfunction of $\fraclaplacian$ with zero exterior Dirichlet condition in $\R^n\setminus \Omega$ (which is stable).
	Note that, in view of the arguments in \Cref{Remark:StrictConvexity}, this is the only other possible option if $u\not\equiv 0$ whenever $f$ is (non strictly) convex and $\s \geq 1/2$.
%	
%	, at least  if $s\geq 1/2$.
%	Indeed, as mentioned in \Cref{Remark:StrictConvexity}, if $f$ is (non strictly) convex and $\s \geq 1/2$, then we have that either $u\equiv 0$ (as occurs in the previous proof), or $f(t) = a + \lambda_1 t$ for $t\in (0, \sup_\Omega u)$.
%	In the second case, since $f(0) = 0$, we have that  $f(t) =\lambda_1 t$ for $t\in (0, \sup_\Omega u)$ and by the strong maximum principle $u>0$ in $\Omega$.
%	Thus, $u$ is a first eigenfunction of $\fraclaplacian$ with zero exterior Dirichlet condition in $\R^n\setminus \Omega$.
\end{remark}

%%%%%%%%%%%%%%%%%%%%%%%%%%%%%%%%%%%%%%%%%%%%%%%%%%%%
\section{Approximation of stable solutions}
%%%%%%%%%%%%%%%%%%%%%%%%%%%%%%%%%%%%%%%%%%%%%%%%%%%%
\label{Sec:Approximation}

In this section we establish our main results regarding the approximation of stable solutions in $H^\s_{\Omega}$ (by bounded stable solutions).

\begin{proof}[Proof of \Cref{Th:HsAppConvex}]
	We proceed as in \cite{DupaigneFarina}. 
	Note first that if $f'(t) \leq 0$ for all $t\geq -M$, then by convexity we have that $f$ is globally Lipschitz in $[-M,+\infty)$ ---since $ f'(-M) \leq f'(t) \leq 0$ for $t\in [-M,+\infty)$--- and then a standard bootstrap argument (using \Cref{Prop:RegularityLp}, as in the proof of \Cref{Prop:uepsilon} below) yields that $u\in L^\infty(\Omega)$; thus we can take $u_k = u $ and $f_k = f$ for all $k\in \N$.
	As a consequence, in the rest of the proof we will assume that there exists some $t_\bullet>-M$ for which $f'(t_\bullet) > 0$.
	
	Under the previous assumption, by convexity there exists $t_0>0$ such that $f'(t) > 0$ for $t>t_0$.
	We set $k_0 := \lceil t_0 \rceil$, i.e., the smallest integer greater or equal than $t_0$ (hence $k_0 \geq 1$), and for every integer $k\geq k_0$ we define
	\begin{equation}
		\label{Eq:fkApprox}
		f_k(t)  :=
		\beqc{lcc}
		f(t) & \text{ if } & t \leq k,\\
		f(k) + f'(k) (t-k) & \text{ if } & t > k.\\
		\eeqc
	\end{equation}
	Note that each $f_k$ is a $C^1$ convex function which is globally Lipschitz in $[-M,+\infty)$ (since $f'(-M) \leq f_k'(t) \leq f'(k)$ for all $t\geq -M$) and bounded below by $C_\circ := \min_{t \in [-M, k_0]} f(t)$.
	By convexity we have that $f_k(t) \leq f(t) $ for all $t\geq -M$, and clearly $f_k \uparrow f$ pointwise in $[-M,+\infty)$ as $k\to + \infty$. 
	Moreover, if $f$ is nonnegative, all the functions $f_k$ are also nonnegative, and if $f$ is $C^{1,\gamma}$ for some $\gamma$, so are $f_k$. 
	
	In order to build a solution to \eqref{Eq:SemilinearPbHsAppfkConvex} for a nonlinearity $f_k$ as in \eqref{Eq:fkApprox}, we will use the monotone iteration method (see \Cref{Prop:MonotoneIteration}).
	As a subsolution, we will take $\usub := -M$ (recall that $0 \leq f(-M)$ and that $u\geq -M$  in $\R^n$), while as a supersolution we will take $\usup := u$ (note that $u\geq-M$ in $\Omega$, and thus $f_k(u) \leq f(u)$ in $\Omega$).
	Then, since  $f_k$ is globally Lipschitz in $[-M,+\infty)$, we can apply \Cref{Prop:MonotoneIteration} (taking into account \Cref{Remark:MonotoneIteration}) to obtain a solution $u_k$ to \eqref{Eq:SemilinearPbHsAppfkConvex} such that $-M\leq u_k \leq u$ in $\Omega$.
	Note that since $f_k$ is bounded by below and linear at infinity, a standard bootstrap argument ---as that of \Cref{Prop:uepsilon} below, using \Cref{Prop:RegularityLp} taking into account that $u_k \leq u$ in $\Omega$ and that $u\in L^1(\Omega)$--- yields that $u_k$ is bounded in $\Omega$, and hence regular (see \Cref{Remark:BoundedSolutions}); in particular, $u_k \in H^\s_{\Omega}$ (see \Cref{Sec:ExteriorCondition} for more details on the role of the exterior condition).
	Furthermore, since $f$ is convex and $f_k' (t) \leq f'(t)$ for $t\geq -M$, it follows that $f_k'(u_k) \leq f_k'(u) \leq f'(u)$ in $\Omega$, and hence $u_k$ is stable.
	
	It remains to show that $u_k \to u$ in $H^\s_\Omega$  as $k\to +\infty$ (up to subsequences).
	To do so, note first that $f_{k} (t) \leq f_{k+1}(t)$ for all $t\geq -M$ and $k\in \N$ and thus $u_{k+1}$ is a supersolution of the problem for $u_k$.
	Hence, $u_k \leq u_{k+1}$ in $\Omega$ and taking the pointwise limit of $u_k$ as $k\to +\infty$ (which exists by monotonicity) we obtain a function $v$ such that $v \leq u $ in $\Omega$.
	Furthermore, by the dominated convergence theorem, $u_k \to v$ in $L^2(\Omega)$ as $k\to +\infty$.
	
	Now, we test the equation for $u_k$ with $u-u_k \in H^\s_{\Omega,0}$ (as in \Cref{Remark:L1WeakSolWhichAreHs} using that $u_k$ is bounded) to obtain
	\begin{equation}
		\langle u_k, u - u_k \rangle_{H^\s_\Omega} = \int_\Omega f_k (u_k) (u - u_k) \d x.
	\end{equation}
	Hence, using that $f_k$ is bounded below by the constant $C_\circ = \min_{t \in [-M, k_0]} f(t)$ which is independent of $k$, we have
	\begin{equation}
		\begin{split}
			[u_k]^2_{H^\s_\Omega} &= \langle u_k, u_k \rangle_{H^\s_\Omega} = \langle u_k, u \rangle_{H^\s_\Omega}  - \int_\Omega f_k (u_k) (u - u_k) \d x \\
			& \leq \langle u_k, u \rangle_{H^\s_\Omega}  - C_\circ \int_\Omega  (u - u_k) \d x \\
			&\leq \dfrac{1}{2}\langle u_k, u_k \rangle_{H^\s_\Omega} + \dfrac{1}{2}\langle u, u \rangle_{H^\s_\Omega}  + |C_\circ| \norm{u}_{L^1(\Omega)} + |C_\circ| \norm{u_k}_{L^1(\Omega)}.
		\end{split}
	\end{equation} 
	Thus, since $-M \leq u_k \leq u$ in $\Omega$, it follows that $ \norm{u_k}_{L^1(\Omega)}$ and $ \norm{u_k}_{L^2(\Omega)}$ are bounded independently of $k$, and therefore the sequence $u_k$ is uniformly bounded in $H^\s_{\Omega}$.
	By the reflexivity of this space, there exists a subsequence $k_j$ for which $u_{k_j} \rightharpoonup v$ in $H^\s_{\Omega}$.
	To show that we have indeed strong convergence, we test the equation for $u_{k_j}$ with $v - u_{k_j} \in H^\s_{\Omega,0}$ and similarly as before we obtain
	\begin{equation}
		\norm{u_{k_j}}^2_{H^\s_\Omega} \leq \norm{v}^2_{H^\s_\Omega}  + 2 |C_\circ| \int_\Omega  (v - u_{k_j}) \d x.
	\end{equation}
	Hence, using that $u_{k_j} \to v$ in $L^2(\Omega)$ we obtain
	\begin{equation}
		\limsup_{k_j \to +\infty} 	\norm{u_{k_j}}^2_{H^\s_\Omega} \leq \norm{v}^2_{H^\s_\Omega}  .
	\end{equation}
	This together with the weak convergence yields that $u_{k_j} \to v$ in $H^\s_{\Omega}$ as $k_j \to +\infty$ (see \cite[Proposition~3.32]{Brezis}).
	
	We shall prove now that $v$ solves the same problem as $u$ (in the $L^1$-weak sense).
	To establish this, it suffices to take the limit $k_j\to \infty$ in the $L^1$-weak formulation of the equation for $u_{k_j}$ using that $f_{k_j} (u_{k_j}) \to f(v)$ in $L^1(\Omega)$.
	This last statement follows by dominated convergence, since 
	\begin{equation}
		C_\circ \leq f_{k_j} (u_{k_j}) \leq f(u_{k_j}) = f(u_{k_j}) \chi_{\{v_{k_j}\leq k_0\}} + f(u_{k_j}) \chi_{\{v_{k_j}> k_0\}} \leq \sup_{t\in [-M,k_0]} f(t) + |f(u)|.
	\end{equation}

	To conclude, since $v$ is a stable $L^1$-weak solution to \eqref{Eq:SemilinearProblem} with $-M\leq v\leq u$, \Cref{Prop:UniquenessIfNotEig} yields that  $u \equiv v$, concluding the proof.
\end{proof}

\begin{remark}
	\label{Remark:StrictConvexityApprox}
	Note that the strict convexity hypothesis of $f$ has been used only in the last point of the proof, to use \Cref{Prop:UniquenessIfNotEig}.
	Nevertheless, \Cref{Th:HsAppConvex} could be generalized to (non strictly) convex nonlinearities in some cases.
	Indeed, if $f$ is piecewise affine\footnote{That is, defined piecewise in a finite number of intervals and affine in each of these intervals.}, since $f$ is globally Lipschitz it follows that $u$ is bounded (this can be seen using a bootstrap argument); we can take then $f_k = f$ and $u_k = u$ for every $k\in \N$.
	In particular, note that if $\s\geq 1/2$ we showed in \Cref{Remark:StrictConvexity} that if $f$ is not strictly convex and uniqueness does not hold, then necessarily $f(t) = a + \lambda_1 t$ for all $t\geq -M$ (and thus also the approximation result holds).
	
%	 (see the argument at the end of \Cref{Remark:StrictConvexity})
%	Otherwise, if $f$ is not affine then uniqueness of solutions holds as well if $\s\geq 1/2$ (as explained in \Cref{Remark:StrictConvexity}).
\end{remark}

In the rest of the section, we establish another approximation result in which the nonlinearity $f$ (which we additionally assume to be nonincreasing) is replaced by $(1-\varepsilon_k) f$ for some sequence $\varepsilon_k \downarrow 0$.
Note that this approximation reminds of the problem of the extremal solution (see \cite[Section~1.1]{CabreSanzPerela-HalfLaplacian}).
For the sake of simplicity, and since \Cref{Th:HsAppConvex} already gives an approximation result for more general stable solutions, we will consider here nonnegative stable solutions which vanish outside $\Omega$.
Nevertheless, the obvious modifications in the proofs establish the result for more general exterior data.
We will also assume that $f(0)>0$, since otherwise our classification result (\Cref{Th:Classiff(0)=0}) yields that $u\equiv 0$ in $\R^n$.

\begin{theorem}
	\label{Th:HsAppEps}
	Let $\Omega\subset \R^n$ be a bounded smooth domain, and let $f\in C^1(\R)$ be a nondecreasing strictly convex function with $f(0) > 0$.
	Let $u$ be a stable $L^1$-weak solution to	
	\begin{equation}
		\label{Eq:SemilinearPbHsApp}
		\beqc{\PDEsystem}
		\fraclaplacian u & = & f(u)& \text{ in } \Omega,\\
		u & = & 0 & \text{ in } \R^n\setminus \Omega,
		\eeqc
	\end{equation}
	such that $u\in H^\s_\Omega$.
	
	Then, there exists a sequence $\{\varepsilon_k\}_{k \in \N} \subset [0,1)$ such that $\varepsilon_k \downarrow 0$ as $k\to +\infty$, and a sequence $\{u_k\}_{k \in \N}$ of bounded stable solutions to
	\begin{equation}
		\label{Eq:SemilinearPbHsAppEpsilon}
		\beqc{\PDEsystem}
		\fraclaplacian u_k & = & (1-\varepsilon_k) f(u_k)& \text{ in } \Omega,\\
		u_k & = & 0 & \text{ in } \R^n\setminus \Omega,
		\eeqc
	\end{equation}
	such that $0 \leq u_k \leq u$ in $\R^n$, and such that $u_k \to u$ in $H^\s_\Omega$  as $k\to +\infty$.
\end{theorem}

To establish \Cref{Th:HsAppEps} we need two preliminary results, \Cref{Lemma:KatoIneq} and \Cref{Prop:uepsilon}. 
The first one is a general statement of what is usually known as concave truncation method.\footnote{Sometimes referred to as Kato's inequality when $\Phi (u) = - u^+$.}
We include the detailed proof for $L^1$-weak solutions and with nontrivial exterior data for the sake of completeness.

\begin{lemma}
	\label{Lemma:KatoIneq}
	Let $\Omega\subset \R^n$ a bounded smooth domain.
	Given $h \in L^1(\Omega, d_\Omega^\s \dx)$ and $g$ as in \eqref{Eq:Assumptionsg}, let $u$ be the $L^1$-weak solution to 
	\begin{equation}
		\beqc{\PDEsystem}
		\fraclaplacian u & = & h& \text{ in } \Omega,\\
		u & = & g & \text{ in } \R^n\setminus \Omega.
		\eeqc
	\end{equation}
	
	Then, for every concave function $\Phi \in C^1(\R)$ such that $\Phi'$ is bounded, we have $\Phi(u) \in L^1(\Omega)$ and $\fraclaplacian \Phi(u) \geq \Phi'(u) h$ in $\Omega$ in the $L^1$-weak sense, that is,
	\begin{equation}
		\int_\Omega  \Phi(u) \fraclaplacian \varphi \d x +  \int_{\R^n \setminus \Omega}\Phi (u) \ncal_\s \varphi \d x  \geq \int_\Omega \Phi'(u) h \varphi \d x 
	\end{equation}
	for every $\varphi$ such that $\varphi$ and $\fraclaplacian \varphi$ are bounded in $\Omega$, and such that $\varphi \geq 0$  in $\Omega$ and $\varphi \equiv 0$  in $\R^n\setminus \Omega$.
\end{lemma}

\begin{proof}
	Let $h_k\in C^\infty_c(\Omega)$ be a sequence of functions such that $h_k \to h$ in  $L^1(\Omega, d_\Omega^\s \dx)$, and define $u_k$ as the solution to 
	\begin{equation}
		\beqc{\PDEsystem}
		\fraclaplacian u_k & = & h_k& \text{ in } \Omega,\\
		u_k & = & g & \text{ in } \R^n\setminus \Omega.
		\eeqc
	\end{equation}
	By the regularity of the fractional Laplacian (see \cite{RosOtonSerra-Regularity, AudritoRosOton}), it follows that $u_k \in C^\infty(\Omega)\cap C^{\alpha_0} (\overline{\Omega})$ with $\alpha_0$ as in \eqref{Eq:Assumptionsg}, and therefore the equation holds pointwise.
	Note that the functions $u_k$ are uniformly bounded in $L^1(\Omega)$ and converge to $u$ in $L^1(\Omega)$ (by uniqueness of $L^1$-weak solutions).
	
	Now, assume for a while that $\Phi$ is smooth.
	Since it is concave, we have $\Phi(a )- \Phi(b) \geq \Phi'(a) (a- b)$ for all $a,b \in \R$, and thus it readily follows that 
	\begin{equation}
		\label{Eq:FracLapConcave}
		\fraclaplacian \Phi (u_k) \geq \Phi'(u_k) \fraclaplacian u_k = \Phi'(u_k) h_k \quad \text{ in } \Omega. 
	\end{equation}
	Multiplying the above inequality by a function $\varphi$ being as in the statement of the result, and integrating by parts the fractional Laplacian (using that $\varphi \equiv 0$ in $\R^n \setminus \Omega$), we have
	\begin{equation}
		\label{Eq:KatoWeakProof}
	\int_\Omega 	 \Phi (u_k) \fraclaplacian\varphi \d x + \int_{\R^n \setminus \Omega}\Phi (u_k) \ncal_\s \varphi \d x  \geq \int_\Omega \Phi'(u_k) h_k \varphi \d x.
	\end{equation}
	Note that we have assumed that $\Phi$ is smooth in order to compute the fractional Laplacian in \eqref{Eq:FracLapConcave} ---we require $\Phi(u_k)$ to be $C^\beta$ for $\beta > 2\s$.
	To avoid this issue if $\Phi$ is only $C^1$, we can just simply take $\Phi_\epsilon$ being a smooth concave approximation\footnote{Note that the convolution of a convex function with a standard mollifier provides a smooth approximation which is convex, see \cite{Ghomi-convex}.} of $\Phi$ in the range of $u_k$ (which is smooth and hence bounded) and carry out the previous argument for such $\Phi_\epsilon$. 
	Then, by letting $\epsilon\to 0$ and using that $\Phi$ is $C^1$ and that $\Phi'$ is bounded, we obtain \eqref{Eq:KatoWeakProof}.

	Now, since $\Phi'$ is bounded, it follows that $|\Phi(t)|  \leq C |t| + |\Phi (0)|$ for all $t$ in $\R$, and therefore $\Phi(u_k)$ are uniformly bounded in $L^1(\Omega)$.
	Hence, there exists a sequence $k_j$ such that $\Phi(u_{k_j}) \to \Phi(u)$ in $L^1(\Omega)$ as $k_j \to +\infty$ (to see this, we use convergence a.e., the continuity of $\Phi$, and dominated convergence).
	Similarly, we also have that $\Phi'(u_{k_j}) h_{k_j} \to \Phi'(u) h$ in $L^1(\Omega, d_\Omega^\s \d x)$ ---up to a subsequence which we relabel as $k_j$ abusing of notation.
	Thus, taking the limit $k_j \to +\infty$ in \eqref{Eq:KatoWeakProof} and noting that $\Phi(u_{k_j}) = \Phi(g) = \Phi(u)$ in $\R^n \setminus \Omega$ for all $k_j$, the desired conclusion follows.
\end{proof}

With the previous tool at hand we can now establish the second result needed in the proof of \Cref{Th:HsAppEps}.
In its proof, which follows the same lines of \cite[Theorem 3.2.1]{Dupaigne}, the construction of the approximate solutions is explicitly described.
The main idea is to construct the solution $u_\varepsilon$ using monotone iteration, taking $0$ as a subsolution, and $\Phi_\varepsilon(u)$,  for some concave function $\Phi_\varepsilon$, as a supersolution.

\begin{lemma}
	\label{Prop:uepsilon}
	Let $\Omega \subset \R^n$ be a bounded smooth domain and let $f\in C^1(\R)$ be a nondecreasing convex function with $f(0) > 0$.
	Assume that $u$ is a stable $L^1$-weak solution to \eqref{Eq:SemilinearProblemZero}.
	
	Then, for every $\varepsilon\in (0,1)$ there exists a stable solution $u_\varepsilon$ to 
	\begin{equation}
		\label{Eq:SemilinearProblemEpsilon}
		\beqc{\PDEsystem}
		\fraclaplacian u_\varepsilon & = & (1 - \varepsilon)  f(u_\varepsilon) & \text{ in } \Omega,\\
		u_\varepsilon & = & 0 & \text{ in } \R^n\setminus \Omega,
		\eeqc
	\end{equation}
	which is bounded (and hence it is a pointwise solution).
\end{lemma}

\begin{proof}

	For $\delta \in (0,1)$, let $\Phi_\delta$ be the solution of the initial value problem
	\begin{equation}
		\label{Eq:ODEPhiDelta}
		\beqc{\PDEsystem}
		\Phi_\delta(t)' f(t) & = & (1 - \delta)  f(\Phi_\delta(t)) & \text{ for } t >0,\\
		\Phi_\delta (0) & = & 0, &
		\eeqc
	\end{equation}
	defined implicitly by
	\begin{equation}
		\label{Eq:PhiIntegrals}
		\int_0^{\Phi_\delta(t)} \dfrac{\df \tau }{f(\tau)} = (1-\delta) \int_0^t \dfrac{\df \tau }{f(\tau)} \quad \text{ for } t> 0.
	\end{equation}
	From \eqref{Eq:ODEPhiDelta} and \eqref{Eq:PhiIntegrals}, using that $f$ is positive, it follows that $\Phi_\delta$ is strictly increasing and that $0 \leq \Phi_\delta(t) \leq t$ for all $t\geq 0$ (and thus, since $f$ is nondecreasing, we have $\Phi_\delta'(t) \leq 1 - \delta$ for $t\geq 0$). 
	Now, using this last assertion, the ODE for $\Phi_\delta$, and that $f$ is nonnegative, nondecreasing, and convex, we obtain
	\begin{equation}
		\begin{split}
			\Phi_\delta''(t) & =  (1-\delta) \dfrac{f'(\Phi_\delta(t)) \Phi_\delta' (t) f(t) - f(\Phi_\delta(t)) f'(t)}{f(t)^2} \\
			&= (1-\delta) \dfrac{f(\Phi_\delta(t))}{f(t)^2} \Big ( f'(\Phi_\delta(t)) (1- \delta) -  f'(t) \Big ) \\
			&\leq (1-\delta) \dfrac{f(\Phi_\delta(t))}{f(t)^2} \Big ( f'(\Phi_\delta(t)) -  f'(t) \Big ) \\
			&\leq 0,
		\end{split}
	\end{equation}
	showing that $\Phi_\delta$ is concave.
	Therefore, we can apply \Cref{Lemma:KatoIneq} to see that $\Phi_\delta(u) \in L^1(\Omega)$ and that 
	\begin{equation}
		\int_\Omega  \Phi_\delta(u) \fraclaplacian \varphi \d x +   \int_{\R^n \setminus \Omega}\Phi_\delta (u) \ncal_\s \varphi \d x  \geq \int_\Omega \Phi_\delta'(u) f(u) \varphi \d x = \int_\Omega (1 - \delta) f(\Phi_\delta(u)) \varphi \d x,
	\end{equation}
	for every $\varphi$ such that $\varphi$ and $\fraclaplacian \varphi$ are bounded in $\Omega$, and such that $\varphi \geq 0$  in $\Omega$ and $\varphi \equiv 0$  in $\R^n\setminus \Omega$.
	Hence, recalling that $u\equiv 0$ in $\R^n\setminus \Omega$ and that $\Phi_\delta (0)=0$ we have proved that $\Phi_\delta(u)$ is an $L^1$-weak supersolution to 
	\begin{equation}
		\label{Eq:SemilinearProblemDelta}
		\beqc{\PDEsystem}
		\fraclaplacian w & = & (1 - \delta)  f(w) & \text{ in } \Omega,\\
		w & = & 0 & \text{ in } \R^n\setminus \Omega.
		\eeqc
	\end{equation}
	Now, using that $f(0) > 0$, it follows that $0$ is an $L^1$-weak subsolution to \eqref{Eq:SemilinearProblemDelta}.
	Therefore, by the monotone iteration method (see \Cref{Prop:MonotoneIteration}) we obtain a solution $u_{\delta,1}$ to \eqref{Eq:SemilinearProblemDelta} with $0 \leq u_{\delta,1} \leq \Phi_\delta (u)\leq u$  in $\Omega$.
	
	Note that if $\Phi_\delta$ is bounded,\footnote{This is equivalent to say that $ \int_0^{+\infty} \frac{\df \tau}{f(\tau)} = +\infty,$ and thus the boundedness or not of   $\Phi_\delta$ is completely independent of $\delta\in (0,1)$.}
	we already have the desired bounded solution setting $\delta = \varepsilon$ and  $u_\varepsilon := u_{\varepsilon,1}$.
	Otherwise, we need to iterate the previous procedure and, as we will see, in each step we will improve the integrability of the solution, obtaining after a finite number of iterations a bounded solution.

	Repeating the previous arguments with $u_{\delta,1}$ playing the role of $u$ (i.e., taking $\Phi_\delta (u_{\delta,1} )$ as a supersolution), and replacing $f$ by $(1-\delta)f$ (noticing that this does not affect the definition of $\Phi_\delta$), we obtain a function $u_{\delta,2}$ which is an $L^1$-weak solution to 
	\begin{equation}
		\label{Eq:SemilinearProblemDelta2}
		\beqc{\PDEsystem}
		\fraclaplacian u_{\delta,2} & = & (1 - \delta)^2  f(u_{\delta,2}) & \text{ in } \Omega,\\
		u_{\delta,2} & = & 0 & \text{ in } \R^n\setminus \Omega,
		\eeqc
	\end{equation}
	such that  $0 \leq u_{\delta,2} \leq \Phi_\delta ( u_{\delta,1} ) \leq u$ in $\Omega$.
	Iterating this process, for each positive integer $k$ we obtain 	a function $u_{\delta,k}$ which is an $L^1$-weak solution to 
	\begin{equation}
		\label{Eq:SemilinearProblemDeltak}
		\beqc{\PDEsystem}
		\fraclaplacian u_{\delta,k} & = & (1 - \delta)^k f(u_{\delta,k}) & \text{ in } \Omega,\\
		u_{\delta,k} & = & 0  & \text{ in } \R^n\setminus \Omega,
		\eeqc
	\end{equation}
	such that  $0 \leq u_{\delta,k} \leq \Phi_\delta^k (u)\leq u$ in $\Omega$.
	
	We will prove that there exists a positive integer $k_0$ for which $u_{\delta,k_0}$ is bounded. 
	To do so, we show first that for every $\delta\in (0,1)$ it holds
	\begin{equation}
		\label{Eq:GrowthfPhi_eps}
		f(\Phi_\delta (t))  \leq \dfrac{C_\circ}{\delta} (1 + t) \quad \text{ for all } t \geq 0 
	\end{equation}
	and for some constant $C_\circ$ depending only on $f$.
	To prove \eqref{Eq:GrowthfPhi_eps}, we set
	\begin{equation}
		\phi(t) := \int_0^t \dfrac{\df \tau }{f(\tau)}
	\end{equation}
	and noting that $\phi$ is increasing and concave in $(0,+\infty)$ ---since $\phi'(t) = 1/f(t)$ and $\phi''(t) = - f'(t)/f^2(t)$ for $t>0$---, for every $a,b\in (0,+ \infty)$ with $a \leq b$ we have
	\begin{equation}
		\phi(b) \leq \phi(a) + \phi'(a)(b - a) = \phi(a) + \dfrac{b - a}{f(a)}.
	\end{equation}
	Taking $a = \Phi_\delta (b)$ and using that $\phi(a) = \phi (\Phi_\delta (b)) = (1-\delta) \phi(b)$ ---see \eqref{Eq:PhiIntegrals}---, we obtain
	\begin{equation}
		\phi(b) \leq (1-\delta) \phi(b)+ \dfrac{b - \Phi_\delta (b)}{f(\Phi_\delta (b))},
	\end{equation}
	and therefore 
	\begin{equation}
		\delta f(\Phi_\delta (b)) \leq \dfrac{b - \Phi_\delta (b)}{\phi(b)} \leq \dfrac{b}{\phi(b)}.
	\end{equation}
	Hence, on the one hand, for $b\geq 1$ we have $\phi(b) \geq \phi(1)$, and thus
	\begin{equation}
		f(\Phi_\delta (b))  \leq \dfrac{b}{ \delta \phi(1)}.
	\end{equation}
	On the other hand, for $b\leq 1$ it follows readily that
	\begin{equation}
		f(\Phi_\delta (b))  \leq f(b) \leq  f(1),
	\end{equation}
	and using that $\delta<1$ we establish \eqref{Eq:GrowthfPhi_eps} with $C_\circ = \max\{ f(1), \phi(1)^{-1}\}$.
	
	Once we have proved \eqref{Eq:GrowthfPhi_eps}, let us show that there exists a positive integer $k_0$ (depending only on $n$ and $\s$) such that $u_{\delta,k_0}$ is bounded. 
	First, we notice that, since $u_{\delta,1} \leq \Phi_\delta (u)$, \eqref{Eq:GrowthfPhi_eps} yields
	\begin{equation}
		f(u_{\delta,1}) \leq f(\Phi_\delta (u)) \leq \dfrac{C_\circ}{\delta} (1 + u)
	\end{equation}
	and since $u\in L^1(\Omega)$, by \Cref{Prop:RegularityLp}~\textit{(i)} it follows that $u_{\delta,1}\in L^{p_1}(\Omega)$ for some $p_1 < \frac{n}{n - 2\s} $.
	Repeating this argument we have
	\begin{equation}
		f(u_{\delta,2}) \leq f(\Phi_\delta (u_{1,\delta})) \leq \dfrac{C_\circ}{\delta} (1 + u_{1,\delta})
	\end{equation}
	and since $u_{\delta,1}\in L^{p_1}(\Omega)$, by \Cref{Prop:RegularityLp}~\textit{(ii)} we get that $u_{\delta,2}\in L^{p_2}(\Omega)$ for $p_2 := \frac{n p_1}{n - 2 \s p_1}$.
	Iterating this procedure, it is easy to see that after a finite number of iterations we reach a power $p_{k_0} > n/(2\s)$, and thus \Cref{Prop:RegularityLp}~\textit{(iii)} yields that $u_{\delta, k_0}\in L^\infty(\Omega)$.

	Taking $\delta = \delta(\varepsilon)$ such that $(1 - \delta)^{k_0} = (1-\varepsilon)$, we have obtained the desired  solution by setting $u_\varepsilon := u_{\delta(\varepsilon),k_0}$.
	Finally, by convexity we have that $f'(u_\varepsilon) \leq f'(u)$ in $\Omega$ and thus $u_\varepsilon$ is stable.
\end{proof}

Having proved the previous result, we conclude the section by establishing \Cref{Th:HsAppEps}.

\begin{proof}[Proof of \Cref{Th:HsAppEps}]
	For each $\varepsilon\in (0,1)$, let $u_\varepsilon$ be the bounded stable solution given by \Cref{Prop:uepsilon}. 
	By construction, $0 \leq u_\varepsilon \leq u$ in $\R^n$ for all $\varepsilon \in (0,1)$, and therefore $\norm{u_\varepsilon}_{L^2(\Omega)} \leq \norm{u}_{L^2(\Omega)}$. 
	Now, since $u_\varepsilon$ is a bounded solution (and thus regular enough), we can multiply the equation that it satisfies by $u_\varepsilon$ and integrate by parts to obtain
	\begin{equation}
		\label{Eq:H1ApproxProof}
		[u_\varepsilon]^2_{H^\s_\Omega} = (1 - \varepsilon) \int_{\Omega} f(u_\varepsilon) u_\varepsilon \d x \leq \int_{\Omega} f(u) u \d x =  [u]^2_{H^\s_\Omega} ,
	\end{equation}
	where we have used the monotonicity of $f$ and the fact that $u\in H^\s_{\Omega,0}$ (and thus one can perform one integration by parts in the $L^1$-weak formulation, see \Cref{Remark:L1WeakSolWhichAreHs}).
	Hence, $u_\varepsilon$ are uniformly bounded in $H^\s_\Omega$, and therefore there exists a sequence $\{\varepsilon_k\}_{k \in \N}$ with $\varepsilon_k \to 0$ as $k \to +\infty$ such that $u_{\varepsilon_k} \rightharpoonup v$ in $H^\s_\Omega$ for some $v \in H^\s_\Omega$.
	Furthermore, since $f$ is nondecreasing, $f(u_\varepsilon) \leq f(u)$ and thus, since $f(u) \in L^1(\Omega)$, by dominated convergence we have $f(u_{\varepsilon_k}) \to f(v)$ in $L^1(\Omega)$.
	Thus, taking the limit $k\to + \infty$ in the weak formulation of the equation for $u_{\varepsilon_k}$ we obtain that $v\in H^\s_\Omega$ is an $L^1$-weak solution to \eqref{Eq:SemilinearPbHsApp}.
	Hence, by \Cref{Prop:UniquenessIfNotEig}, $u = v$.
	
	Finally, we show that $u_{\varepsilon_k} \to u$ in $H^\s_\Omega$ as $k \to +\infty$.
	By the first equality in \eqref{Eq:H1ApproxProof}, this is equivalent to showing that $(1-\varepsilon_k) f(u_{\varepsilon_k})u_{\varepsilon_k}$ converges to $f(u) u$ in $L^1(\Omega)$, and this follows from the inequality in \eqref{Eq:H1ApproxProof} by dominated convergence, since $ [u]^2_{H^\s_\Omega}<+\infty$.
\end{proof}

%%%%%%%%%%%%%%%%%%%%%%%%%%%%%%%%%%%%%%%%%%%%%%%%%%%%
\section{Application: regularity of stable solutions in low dimensions}
%%%%%%%%%%%%%%%%%%%%%%%%%%%%%%%%%%%%%%%%%%%%%%%%%%%%
\label{Sec:Applications}

In this section we show how to combine the approximation results of the previous section with universal a priori estimates for stable solutions to establish the regularity of energy solutions which are stable.
In particular, we will use the following a priori estimate from \cite{CabreSanzPerela-HalfLaplacian}, which provides a universal Hölder estimate in low dimensions for half-Laplacian semilinear equations.
Here and through the paper, for $\s\in (0,1)$ we denote by $L^1_{\s} (\R^n)$ the space of measurable functions for which the norm
\begin{equation}
	\label{Eq:DefL1s}
	\norm{w}_{L^1_{\s} (\R^n)} := \int_{\R^n} \dfrac{ |w(x)|}{(1 + |x|^2)^{\frac{n + 2\s}{2}}} \d x
\end{equation}
is finite.

\begin{theorem}[\cite{CabreSanzPerela-HalfLaplacian}]
	\label{Th:Holder}
	Let $n\geq 1$ and $u\in C^2(B_1)\cap L^1_{1/2}(\R^n)$ be a stable solution to $(-\Delta)^{1/2}u = f(u)$ in $B_1 \subset \R^n$, where $f$ is a nonnegative convex $C^{1,\gamma}$ function for some $\gamma >0$. 
	
	Then, 
	\begin{equation}
		\label{Eq:Holder}
		\norm{u}_{C^\beta (\overline{B}_{1/2})} \leq C   \norm{u}_{L^1_{1/2} (\R^n)}  \qquad \text{if }\  1 \leq n\leq 4,
	\end{equation}
	for some dimensional constants $\beta>0$ and $C$.
\end{theorem}

From the previous universal estimate and using \Cref{Th:HsAppConvex} we establish the regularity of $H^{1/2}_\Omega$ stable solutions in low dimensions.

\begin{proof}[Proof of \Cref{Th:HolderWeak}]
	First, by \Cref{Th:HsAppConvex},
	there exists a sequence $\{u_k\}_{k \in \N}$ of bounded stable solutions to
	\begin{equation}
		\label{Eq:SemilinearPbHsAppfkConvexBis}
		\beqc{\PDEsystem}
		\halflaplacian u_k & = & f_k(u_k)& \text{ in } \Omega,\\
		u_k & = & g& \text{ in } \R^n\setminus \Omega,
		\eeqc
	\end{equation}
	where $\{f_k\}_{k \in \N} \subset C^{1,\gamma}(\R)$ is a sequence of globally Lipschitz nonnegative convex nonlinearities converging pointwise to $f$ in $\R$.
	Moreover, $-M \leq u_k \leq u$  in $\R^n$ for some $M$ and $u_k \to u$ in $H^{1/2}_\Omega$  as $k\to +\infty$.

	Now, given a point $x_0 \in \Omega$, let $R>0$ be such that $B_R(x_0) \subset\subset \Omega$.
	On the one hand, since each $u_k$ is bounded in $\Omega$ and $f_k\in C^{1,\gamma}(\R)$, by standard regularity for the half-Laplacian it follows that $u_k\in C^2(B_R(x_0))$ (see Corollaries~2.3 and 2.5 in \cite{RosOtonSerra-Regularity}).
	On the other hand, by the assumption \eqref{Eq:Assumptionsg}, we get that $u\in L^1_{1/2}(\R^n)$ ---see \Cref{Sec:ExteriorCondition}.
	Thus, since $1 \leq n\leq 4$, the universal estimate \eqref{Eq:Holder} yields
	\begin{equation}
		\norm{u_k}_{C^\beta (\overline{B_{R/2}(x_0)})} \leq C   \norm{u_k}_{L^1_{1/2} (\R^n)} \leq C   \norm{u}_{L^1_{1/2} (\R^n)}
	\end{equation}
	for some universal $\beta\in (0,1)$ depending only on $n$, and some constant $C$ depending on $n$, $M$, and $R$.
	As a consequence, by the compact embedding of Hölder spaces, there exists a subsequence (which we relabel as $\{u_k\}_{k \in \N}$) such that $u_k \to v$ for some $v\in C^{\tilde\beta} (\overline{B_{R/2}(x_0)})$ with $\tilde\beta< \beta$.
	Moreover, we already know that $u_k \to u$ a.e. in $\overline{B_{R/2}(x_0)}$, and thus $v = u$ in $\overline{B_{R/2}(x_0)}$.
	Since this holds for every point $x_0\in \Omega$, we obtain that $u\in C^{\tilde\beta} (\Omega)$.
	The conclusion then follows from standard interior estimates for the half-Laplacian using the regularity of~$f$.	
\end{proof}

%%%%%%%%%%%%%%%%%%%%%%%%%%%%%%%%%%%%%%%%%%%%%%%%%%%%
\section{Unbounded $L^1$-weak stable solutions in low dimensions}
%%%%%%%%%%%%%%%%%%%%%%%%%%%%%%%%%%%%%%%%%%%%%%%%%%%%
\label{Sec:Counterexamples}

In this section we show that the assumption on $u$ belonging to $H^\s_{\Omega}$ is crucial in our statements.
To do it, we provide a counterexample for our approximation result which is an $L^1$-weak solution not belonging to the energy space.

Recall (see \cite{RosOton-Gelfand}) that the fractional Laplacian of a function of the form $|x|^{-\beta}$ is given by
\begin{equation}
	\label{Eq:FracLapPower}
	\fraclaplacian |x|^{-\beta} = 2^{2\s}\,
	\frac{\Gamma\left(\frac{\beta+2\s}{2}\right)\Gamma\left(\frac{n-\beta}{2}\right)}
	{\Gamma\left(\frac{n-\beta-2\s}{2}\right)\Gamma\left(\frac{\beta}{2}\right)} |x|^{-\beta-2\s} \,, \quad \text{ for } \beta\in(0,n-2\s).
\end{equation}
For $p>1$, we consider the function
\begin{equation}
	u_p(x) := |x|^{-\frac{2\s}{p - 1}} - 1,
\end{equation}
and, for convenience, we define as well $g_p :\R^n \setminus B_1 \to \R$ by 
\begin{equation}
	g_p(x) := |x|^{-\frac{2\s}{p - 1}} - 1, \quad \text{ for } x\in \R^n \setminus B_1.
\end{equation}

From \eqref{Eq:FracLapPower}, a formal computation shows that
\begin{equation}
	\fraclaplacian u_p (x) = \Lambda_{n,\s,p} |x|^{-\frac{2\s}{p - 1} - 2\s} =  \Lambda_{n,\s,p} |x|^{-\frac{2\s p }{p - 1}} =  \Lambda_{n,\s,p} (1 + u_p(x))^p
\end{equation} 
where
\begin{equation}
	\Lambda_{n,\s,p} := 2^{2\s}\,
	\frac{\Gamma\left(\frac{\s}{p - 1} + \s \right)\Gamma\left(\frac{n}{2} - \frac{\s}{p - 1}\right)}
	{\Gamma\left(\frac{n-2\s}{2} - \frac{\s}{p - 1}\right)\Gamma\left(\frac{\s}{p - 1}\right)}.
\end{equation}
Therefore, if $p> n/(n-2\s)$, we have that $(1 + u_p)^p \in L^1(B_1)$ $u_p$ and thus $u_p$ is an $L^1$-weak solution to the problem
\begin{equation}
	\label{Eq:SemilinearProblemPower}
	\beqc{\PDEsystem}
	\fraclaplacian u & = & \Lambda_{n,\s,p} (1 + u)^p & \text{ in } B_1,\\
	u & = & g_p & \text{ in } \R^n\setminus B_1.
	\eeqc
\end{equation}
In addition, note that if $p \leq (n + 2\s)/(n - 2\s)$, then
\begin{equation}
	|x|^{-\frac{2\s}{p - 1} -\s} \notin L^2(B_1),
\end{equation}
and therefore $u_p \notin H^\s_{B_1}$.
As a summary, we have that for $p$ such that
\begin{equation}
	\label{Eq:PcondL1NotEnergy}
	\dfrac{n}{n - 2\s} < p \leq \dfrac{n + 2\s}{n - 2\s}
\end{equation}
the function $u_p$ is an $L^1$-weak solution to \eqref{Eq:SemilinearProblemPower} not belonging to $H^\s_{B_1}$. 

Let us  now see that, for some values of $p$ satisfying \eqref{Eq:PcondL1NotEnergy}, $u_p$ is stable in $B_1$.
Note that this is equivalent to check whether the inequality
\begin{equation}
	p \Lambda_{n,\s,p} \int_{B_1} \dfrac{\xi^2}{|x|^{2\s}} \d x \leq \seminorm{\xi}_{H^\s_{B_1}}^2
\end{equation}
holds for every smooth function $\xi \in C^\infty_c(B_1)$.
Recall that fractional Hardy inequality reads
$$
H_{n,\s}\int_{\R^n} \dfrac{\xi^2}{|x|^{2\s}} \d x \leq \dfrac{c_{n,\s}}{2} \int_{\R^n} \int_{\R^n} \dfrac{|\xi (x) - \xi (y)|^2}{|x-y|^{n+2\s}} \d x \d y \quad \text{ where } H_{n,\s}=2^{2\s}\frac{\Gamma^2\left(\frac{n+2\s}{4}\right)}{\Gamma^2\left(\frac{n-2\s}{4}\right)}.
$$
Then, $u_p$ is stable in $B_1$ if and only if $p \Lambda_{n,\s,p}  \leq H_{n,\s}$, which  is explicitly written as
\begin{equation}
	\label{Eq:GammasP}
	p \frac{\Gamma\left(\frac{\s}{p - 1} + \s \right)\Gamma\left(\frac{n}{2} - \frac{\s}{p - 1}\right)}
	{\Gamma\left(\frac{n-2\s}{2} - \frac{\s}{p - 1}\right)\Gamma\left(\frac{\s}{p - 1}\right)} \leq \frac{\Gamma^2\left(\frac{n+2\s}{4}\right)}{\Gamma^2\left(\frac{n-2\s}{4}\right)}.
\end{equation}
Now, if we choose
\begin{equation}
	p := \dfrac{n}{n - 2\s} + \varepsilon \quad \textrm{ for } \varepsilon \leq 2\s /(n - 2\s),
\end{equation}
then \eqref{Eq:GammasP} is satisfied if $\varepsilon$ is small enough.
To check this last assertion, note that
\begin{equation}
	\frac{n-2\s}{2} - \frac{\s}{p - 1} = \frac{n-2\s}{2}  -  \dfrac{1}{\frac{2}{n - 2\s} + \frac{\varepsilon}{\s}} = \varepsilon \dfrac{n - 2\s}{\frac{4\s}{n - 2\s} + 2\varepsilon}\end{equation}
and thus
\begin{equation}
	\Gamma\left(\frac{n-2\s}{2} - \frac{\s}{p - 1}\right) \to +\infty \quad \text{ as } \varepsilon \to 0.
\end{equation}
From this, and noticing that the other expressions  appearing in $\Lambda_{n,\s,p}$ which involve the Gamma function are uniformly controlled as $\varepsilon \to 0$, we see that for $\varepsilon$ small enough \eqref{Eq:GammasP} holds, and thus $u_p$ is stable.

As a conclusion, we have obtained an unbounded $L^1$-weak solution which does not belong to the energy space but which is stable.
This provides a counterexample of our approximation result (\Cref{Th:HsAppEps}) for stable solutions if we do not assume that the solution belongs to the energy space.
Indeed, if $u_p$ could be approximated by bounded stable solutions, in view of the regularity results of \cite{CabreSanzPerela-HalfLaplacian} for $\s =1/2$ and $1 \leq n \leq 4$, and arguing as in the proof of \Cref{Th:HolderWeak}, it would follow that $u_p$ is bounded in $B_{1/2}$, arriving at a contradiction.

\appendix

\section{An ``intermediate value theorem'' in fractional Sobolev spaces }
\label{Sec:IVTh}

In this appendix we recall an ``intermediate value theorem'' for functions in the fractional Sobolev space $W^{\s,p}$, with $p\s \geq 1$, which has been used in the discussion of \Cref{Remark:StrictConvexity}. 
In this result we denote by $|E|$ the $n$-dimensional Lebesgue measure of a measurable set $E$.

\begin{proposition}
	\label{Prop:IVT}
	Let $u\in W^{s,p} (B_1)$ be such that $|\{u \leq a\}\cap B_1|>0$ and $|\{u \geq b\}\cap B_1|>0$ for some $a,b\in \R$ with $a< b$.
	If $sp\geq 1$, then $|\{a < u < b\}\cap B_1|>0$.
\end{proposition}

\begin{proof}
	Assume by contradiction that $|\{a < u < b\}\cap B_1| = 0$.
	Then, setting $E:= \{u \leq a\}\cap B_1$ we have that $B_1 \setminus E =  \{u \geq b\}\cap B_1$ up to a set of measure zero.
	Therefore, if we consider the function $u_{a,b} := \min \{ \max \{ u , a\}, b \}$, since $|u_{a,b}(x) - u_{a,b}(y)| \leq |u(x) - u(y)|$ for all $x,y\in B_1$, it follows that $u_{a,b}\in W^{\s,p} (B_1)$, and thus we have
	\begin{equation}
		 2 c_{n, \s,p} (b - a)^p\int_E \int_{B_1 \setminus E} \dfrac{\df x \, \df y}{|x-y|^{n + sp}} =  [u_{a,b} ]^p_{W^{\s,p}(B_1)} < +\infty .
	\end{equation}
	Then, since $sp \geq 1$ it holds that 
	\begin{equation}
		\int_E \int_{B_1 \setminus E} \dfrac{\df x \, \df y}{|x-y|^{n + 1}}  < +\infty 
	\end{equation}
	 and then by \cite[Corollary~2]{Brezis-ConstantFunctions} it follows that either $|E|= 0$ or $|B_1 \setminus E|=0$, a contradiction.
\end{proof}

\section{Some comments on the exterior condition in the notions of solution}
\label{Sec:ExteriorCondition}

In this appendix we discuss with more details which are the assumptions on the exterior condition $g$ needed to define each type of solution considered in this article.
Our assumptions concern its regularity near $\partial\Omega$ and its growth at infinity.
On the one hand, we assume that $g$ satisfies 
\begin{equation}
	\label{Eq:Holderg}
	|g(x) - g(z) | \leq C_0 |x-z|^\alpha  \quad  \text{ for all }  z\in \partial \Omega  \text{ and } x \in \R^n \setminus \Omega \text{ such that } d_\Omega (x) \leq C_1
\end{equation}
for some positive constants $C_0$ and $C_1$, and some exponent $\alpha\in (0,1)$.
On the other hand, we assume that
\begin{equation}
	\label{Eq:Growthg}
	|g(x)| \leq C_0 (1 + |x|^\gamma) \quad  \text{ for all } x \in \R^n \setminus \Omega
\end{equation}
for some $\gamma \geq 0$.
These pointwise assumptions can be also considered in an integral form, assuming $g$ belonging to an appropriate fractional Sobolev space near $\partial \Omega$, and $g$ in  the right tail space  (see \cite[Section 5.1]{Korvenpaa2016} for some boundary regularity results under these assumptions), but here for simplicity we consider \eqref{Eq:Holderg} and \eqref{Eq:Growthg}.

In the following discussion we will need two estimates for the nonlocal Neumann derivative of a function $\varphi$ which is $C^{\bar{\alpha}}$ across $\partial \Omega$ and has a growth of the form $|\varphi(x)| \leq C (1 + |x|^{\bar{\gamma}})$, for some $\bar{\alpha} \in (0,1)$ and some $\bar{\gamma} \geq 0$.
On the one hand, for $x \in \R^n\setminus \Omega$  with $d_\Omega(x) \leq C$ we have
\begin{equation}
	\label{Eq:EstimateNsNear}
	|\ncal_\s \varphi (x)| \leq C \int_\Omega \dfrac{|\varphi(x) - \varphi(y)|}{|x-y|^{n + 2\s}} \d y \leq C \int_\Omega \dfrac{\d y}{|x-y|^{n + 2\s - \bar{\alpha}}} \d y \leq C d_\Omega (x)^{\bar{\alpha} - 2\s}.
\end{equation}
On the other hand, for $x \in \R^n\setminus \Omega$  with $d_\Omega(x) \geq C$ we have
\begin{equation}
	\label{Eq:EstimateNsFar}
	|\ncal_\s \varphi (x)| \leq \dfrac{C}{1 + |x|^{n + 2\s - \bar{\gamma}}}
\end{equation}

$\bullet$ \textbf{$L^1$-weak solutions}

To properly define $L^1$-weak solutions, the integral
\begin{equation}
	\int_{\R^n \setminus \Omega} g \ncal_\s \varphi \d x
\end{equation}
must be finite for any test function $\varphi$.
Recall that any such $\varphi$ vanishes in $\R^n\setminus \Omega$ (and thus is bounded at infinity) and is $C^\s$ across $\partial \Omega$.
Therefore the above estimates \eqref{Eq:EstimateNsNear} and \eqref{Eq:EstimateNsFar} hold for $\bar{\alpha} = \s$ and $\bar{\gamma} = 0$ respectively.
As a consequence,  $|g \ncal_\s | \leq C|g| d_\Omega^{-\s}$ near $\partial \Omega$ and $|g \ncal_\s | \leq C (1 + |x|^{-n - 2\s})$ at infinity.
We conclude that $L^1$-weak solutions can be well defined for every exterior datum $g \in L^1_\loc (\R^n \setminus \Omega, d_\Omega^{-\s}\d x) \cap L^1_\s(\R^n\setminus \Omega)$.

$\bullet$ \textbf{Energy solutions}

In this context, an energy solution can be seen as an extension (minimizing an $H^\s_{\Omega}$ seminorm) of the function $g$ inside $\Omega$. 
To carry out this minimization process we need to ensure that the set of admissible functions is not empty. 
For this, we define $\overline{g}$ to be the harmonic extension of $g$ inside $\Omega$, which is the smoothest (in $\Omega$) possible extension of $g$, as mentioned in \cite[Lemma~2.2]{AudritoRosOton}.
Indeed, if $g$ satisfies \eqref{Eq:Holderg} for some $\alpha \in (0,1)$, then $\overline{g} \in C^\alpha (\overline{\Omega}) \cap C^\infty (\Omega)$.
We shall show next under which assumptions $\overline{g} \in H^\s_{\Omega}$.

Note first that
\begin{equation}
	\label{Eq:HsNormg}
	\seminorm{\overline{g}}^2_{H^\s_{\Omega}} = \langle \overline{g}, \overline{g} \rangle_{H^\s_\Omega} = \int_\Omega  \overline{g} \, \fraclaplacian \overline{g} \d x + \int_{\R^n \setminus \Omega} \overline{g} \, \ncal_\s \overline{g} \d x,
\end{equation}
and it suffices to check the finiteness of each integral  separately.
On the one hand, the following estimate was proved in \cite[Lemma~2.4]{AudritoRosOton}:
\begin{equation}
	\label{Eq:BoundFracLapg}
	|\fraclaplacian \overline{g} | \leq C d_\Omega^{\alpha - 2\s} \quad \text{ in } \Omega.
\end{equation}
Thus, the first integral in \eqref{Eq:HsNormg} is finite provided that $\alpha  > 2\s -1$.
On the other hand, using estimates \eqref{Eq:EstimateNsNear} and \eqref{Eq:EstimateNsFar} with $\varphi$ replaced by $\overline{g}$ and taking $\bar{\alpha} = \alpha $ and $\bar{\gamma} = \gamma$ we readily see that the second integral in \eqref{Eq:HsNormg} is finite provided that  $\alpha  > 2\s -1$ (to ensure integrability near $\partial \Omega$), and $2\gamma < 2\s$ (to guarantee integrability at infinity).

Summarizing, energy solutions are well defined if one assumes that $g$ satisfies \eqref{Eq:Holderg} and \eqref{Eq:Growthg} with $\alpha  > \max \{ 0, 2\s -1\}$ and $\gamma < \s$ respectively.
Note that since we exclude the set $(\R^n \setminus \Omega)^2$ in the double integral defining the seminorm of $H^\s_{\Omega}$, we do not require any local regularity of $g$ in $\R^n \setminus \Omega$ when we are at a positive distance of $\partial\Omega$.

\begin{remark}
	Condition $\alpha  > \max \{ 0, 2\s -1\}$ illustrates again (as \Cref{Prop:IVT}) the difference of the fractional Sobolev spaces $H^\s_{\Omega}$ depending on whether $\s < 1/2$ or $\s \geq 1/2$.
	Indeed, when $\s < 1/2$ the functions in $H^\s_{\Omega}$ can have jumps across $\partial\Omega$ but the singularity given by $|x-y|^{-n - 2\s}$ is not strong enough so that integrals of the form $\int_\Omega \int_{\R^n \setminus \Omega} |x-y|^{-n - 2\s} \d x \d y$ are infinite (assuming $\partial \Omega$ smooth enough).
	On the other hand, when $\s \geq 1/2$ we need an additional assumption, namely the Hölder regularity of the function $\overline{g}$ across $\partial \Omega$, in order for the difference $\overline{g}(x) - \overline{g}(y)$ to compensate the singularity of $|x-y|^{-n - 2\s}$.
\end{remark}

$\bullet$ \textbf{Pointwise solutions}

In order to have $\fraclaplacian u$ well defined pointwise in $\Omega$, apart from interior regularity we require $u\in L^1_\s(\R^n)$.
Since $u=g$ in $\R^n\setminus \Omega$, it is enough to assume that \eqref{Eq:Growthg} holds fo $\gamma < 2\s$.

%%%%%%%%%%%%%%%%%%%%%%%%%%%%%%%%%
\section{$L^1$-theory for the fractional Laplacian}
%%%%%%%%%%%%%%%%%%%%%%%%%%%%%%%%%
\label{Sec:L1Theory}

In this section we collect the results that are needed through the paper concerning $L^1$-weak solutions.
This setting (with minor modifications and usually for zero exterior condition) has been used in several papers, see \cite{AbdellaouiFernandezLeonoriYounes-CalderonZygmund,BiccariWarmaZuazua-Dirichlet,LeonoriPeralPrimoSoria,RosOtonSerra-Extremal} to name some of them.
For the sake of completeness, we will next provide elementary proofs of the results that we use in the paper, paying especial attention to the role of the nontrivial exterior condition.

We first show how to prove existence and uniqueness of $L^1$-weak solutions to the Dirichlet problem for the fractional Laplacian \eqref{Eq:DirichletPbLinear}.

\begin{proposition}[Existence and uniqueness of $L^1$-weak solutions]
	\label{Prop:ExistenceL1}
	Let $\Omega\subset \R^n$ be a smooth bounded domain, let $h \in L^1(\Omega, d_\Omega^\s \dx)$, and let $g$ satisfy \eqref{Eq:Assumptionsg} for some $\alpha_0 \in (\max\{0, 2\s - 1\},  \s)$ and some positive constant $C_0$.
	
	Then, there exists a unique $L^1$-weak solution $u$ to \eqref{Eq:DirichletPbLinear}	in the sense of \Cref{Def:L1WeakSol}.
	In addition, 
	\begin{equation}
		\label{Eq:L1WeakSolEstimate}
		\norm{u}_{L^1(\Omega)} \leq C \left (\norm{h}_{L^1(\Omega, d_\Omega^\s \dx)} + C_0 \right)
	\end{equation}
	for some constant $C$ depending only $n$, $\s$, $\alpha_0$, and $\Omega$.
\end{proposition}

\begin{proof}
	The uniqueness will follow from the maximum principle (see \eqref{Prop:MaxPple} below), and thus we concentrate on proving the existence.
	We will show first the result for $g\equiv 0$.
	In this case, note that by splitting $h$ into its positive and negative parts $h= h^+ - h^-$, it is enough to prove the existence assuming that $h\geq 0$ a.e. in $\Omega$.
	For such a function $h$, set $h_k := \min \{h, k\}$ and since this is a bounded function, there exists a classical solution $u_k$ to \eqref{Eq:DirichletPbLinear} with $h$ replaced by $h_k$ and with $g\equiv 0$.
	
	Now, take $\eta \in C^\infty_c (\Omega)$ and let $\varphi$ be the solution to 
	\begin{equation}
		\label{Eq:TestFunctionLinearPb}
		\beqc{\PDEsystem}
		\fraclaplacian \varphi & = & \eta& \text{ in } \Omega,\\
		\varphi & = & 0 & \text{ in } \R^n\setminus \Omega.
		\eeqc
	\end{equation}
	Multiplying the previous equation by $u_k$ and integrating in $\Omega$, after using the equation for $u_k$ we obtain
	\begin{equation}
		\label{Eq:L1ExistenceProofWeakSol}
		\int_\Omega u_k \eta \d x = \int_\Omega u_k \fraclaplacian \varphi \d x = \int_\Omega \fraclaplacian u_k \varphi \d x =  \int_\Omega h_k \varphi \d x.
	\end{equation}
	Here we have used that $u_k$ and $\varphi$ are classical solutions of linear equations with bounded right-hand sides, and hence regular enough to perform the usual integration by parts for the fractional Laplacian.
	Note that by standard boundary regularity for the fractional Laplacian (see \cite[Theorem~1.2]{RosOtonSerra-Regularity}), for $ x\in \Omega$ we have $ |\varphi(x) | \leq d_\Omega^\s(x) C \norm{\eta}_{L^\infty(\Omega)}$  for some constant $C$ depending on $\Omega$.
	As a consequence, we obtain the estimate
	\begin{equation}
		\left | \int_\Omega u_k \eta \d x \right |\leq \int_\Omega |h_k| |\varphi| \d x \leq C \norm{\eta}_{L^\infty(\Omega)}  \int_\Omega |h_k| d_\Omega^\s \d x  \leq C \norm{\eta}_{L^\infty(\Omega)}  \norm{h}_{L^1(\Omega, d_\Omega^\s \dx)},
	\end{equation}
	and taking the supremum over all $\eta$ with $\norm{\eta}_{L^\infty(\Omega)}=1$ we get 
	\begin{equation}
		\norm{u_k}_{L^1(\Omega)} \leq C \norm{h}_{L^1(\Omega, d_\Omega^\s \dx)}.
	\end{equation}
	By monotone convergence (note that the sequence $u_k$ is nondecreasing by the maximum principle), as $k\to +\infty$ the sequence $u_k$ converges in $L^1(\Omega)$ to some function $u\in L^1(\Omega)$ satisfying \eqref{Eq:L1WeakSolEstimate} with $C_0= 0$.
	Furthermore, from \eqref{Eq:L1ExistenceProofWeakSol} and by dominated convergence, it follows that $u$ is an $L^1$-weak solution to \eqref{Eq:DirichletPbLinear} with $g\equiv 0$ ---note that we can take $\eta \in L^\infty(\Omega)$ and $\varphi$ the solution to \eqref{Eq:TestFunctionLinearPb}.
	
	Finally, to consider a nonzero exterior condition $g$, by the linearity of the problem it suffices to add, to the previously built function $u$, the solution to
	\begin{equation}
		\beqc{\PDEsystem}
		\fraclaplacian w & = & 0& \text{ in } \Omega,\\
		w & = & g & \text{ in } \R^n\setminus \Omega.
		\eeqc
	\end{equation}
	This solution can be obtained by standard methods (see \cite{DiCastroKuusiPalatucci-Minimizers} for more details) minimizing the $H^\s_\Omega$ seminorm in the convex set $\mathcal{K}_g := \{ v \in H^\s_\Omega \ : \ v - \overline{g} \in H^\s_{\Omega,0}\}$,  where $\overline{g}$ is a smooth extension of $g$ inside $\Omega$ (here we need to assume $\alpha_0> 2\s - 1$ to ensure that $\overline{g} \in H^\s_{\Omega}$, see \Cref{Sec:ExteriorCondition}).
	By the maximum principle (see \cite[Lemma~2.1]{AudritoRosOton}) it follows that
	\begin{equation}
		\norm{w}_{ L^\infty (\Omega)}\leq C C_0
	\end{equation} 
	for some constant $C$ depending only on $n$, $\s$, $\alpha_0$, and $\Omega$.
	From this, estimate \eqref{Eq:L1WeakSolEstimate} follows.
	As an alternative method, one can consider directly the problem for $u-\overline{g}$, whose right-hand side is now $h - \fraclaplacian \overline{g}$ and belongs to $L^1(\Omega)$ if $\alpha > 2 \s - 1$ thanks to estimate \eqref{Eq:BoundFracLapg}, and use the argument for zero exterior data.
\end{proof}

We next present the maximum principle in the $L^1$-weak setting.

\begin{proposition}[Maximum principle]
	\label{Prop:MaxPple}
	Let $\Omega\subset \R^n$ be a smooth bounded domain and let $c\in L^\infty(\Omega)$ be a nonnegative function in $\Omega$.
	Let $u\in L^1_\s(\R^n) \cap L^1_\loc(\R^n\setminus \Omega, d^{-\s}_\Omega \d x)$ be a function satisfying 
	\begin{equation}
		\beqc{\PDEsystem}
		\fraclaplacian u + c u& \geq & 0& \text{ in } \Omega,\\
		u & \geq & 0 & \text{ in } \R^n\setminus \Omega,
		\eeqc
	\end{equation}
	in the $L^1$-weak sense, i.e., such that $u  \geq 0$ a.e. in $\R^n\setminus \Omega$ and
	\begin{equation}
		\int_{\Omega} u \fraclaplacian \varphi \d x +  \int_{\R^n \setminus \Omega} u \ncal_\s \varphi \d x + \int_{\Omega} c u \varphi \d x \geq 0
	\end{equation}
	for every nonnegative function $\varphi$ such that $\varphi$ and $\fraclaplacian \varphi$ are bounded in $\Omega$ and such that $\varphi = 0$ in $\R^n\setminus \Omega$.
	
	Then $u\geq 0$ a.e. in $\Omega$.
\end{proposition}

\begin{proof}
	Take $\eta \in C^\infty_c (\Omega)$ and let $\varphi$ be the solution to 
	\begin{equation}
		\beqc{\PDEsystem}
		\fraclaplacian \varphi + c \varphi & = & \eta& \text{ in } \Omega,\\
		\varphi & = & 0 & \text{ in } \R^n\setminus \Omega.
		\eeqc
	\end{equation}
	Assume that $\eta \geq 0$.
	Then, since $c\geq 0$ in $\Omega$, the standard maximum principle yields $\varphi \geq 0$.
	Plugging this function into the definition of $L^1$-weak solution we obtain
	\begin{equation}
		\int_{\Omega} u \eta \d x + \int_{\R^n \setminus \Omega} u \ncal_\s \varphi \d x \geq 0.
	\end{equation}
	Moreover, since $\varphi = 0$ in $\R^n\setminus \Omega$, then 
	\begin{equation}
		\ncal_\s \varphi (x) = c_{n,\s} \int_\Omega \dfrac{\varphi(x) - \varphi(y)}{|x-y|^{n + 2\s}} \d y =  - c_{n,\s} \int_\Omega \dfrac{\varphi(y)}{|x-y|^{n + 2\s}} \d y  < 0 \quad \text{ for } x \in \R^n\setminus \Omega.
	\end{equation}
	Hence, using that $u\geq 0$ a.e. in $\R^n \setminus \Omega$, we get that
	\begin{equation}
		\int_{\Omega} u \eta \d x \geq 0 \quad  \text{ for every } \eta \in C^\infty_c(\Omega) \text{ such that } \eta \geq 0 \text{ in } \Omega.
	\end{equation}
	Thus, it follows that $u\geq 0$ a.e. in $\Omega$.
\end{proof}

Following a similar argument we can establish a unique continuation principle in the $L^1$-weak setting.

\begin{proposition}[Unique continuation]
	\label{Prop:UniqueContinuation}
	Let $\Omega\subset \R^n$ be a smooth bounded domain, and let $u\in L^1_\s(\R^n) \cap L^1_\loc(\R^n\setminus \Omega, d^{-\s}_\Omega \d x)$ be a function such that $u\geq 0$ a.e. in $\R^n$ and satisfying
	\begin{equation}
		\label{Eq:SuperSHarmonic}
		\fraclaplacian u + c u \geq 0 \quad  \text{ in } \Omega
	\end{equation} 
	in the $L^1$-weak sense and for some $c\in L^\infty(\Omega)$.
	
	If there exists a ball $B\subset \Omega$ such that $u\equiv 0$ a.e. in $B$, then $u\equiv 0$ a.e. in $\R^n$.
\end{proposition}

\begin{proof}
	From the $L^1$-weak formulation of \eqref{Eq:SuperSHarmonic} we have that
	\begin{equation}
		\int_{\Omega} u \fraclaplacian \varphi \d x + \int_{\R^n \setminus \Omega} u \ncal_\s \varphi \d x + \int_{\Omega} c u \varphi \d x  \geq 0
	\end{equation}
	for every nonnegative $\varphi$ such that $\varphi$ and $\fraclaplacian\varphi$ are bounded in $\Omega$, and such that $\varphi = 0$ in $\R^n\setminus \Omega$.	
	Take $\varphi$ being a nonnegative function with compact support in $B$.
	Then, it follows that $\fraclaplacian \varphi < 0$ in $\R^n \setminus B$ and therefore since $u\equiv 0$ a.e. in $B$ we have
	\begin{equation}
		\int_{\Omega} u \fraclaplacian \varphi \d x  = 	\int_{\Omega\setminus B} u \fraclaplacian \varphi \d x \leq 0.
	\end{equation}
	Moreover, since $\varphi$ vanishes outside $B$ we have that  $\ncal_\s \varphi < 0$ in $\R^n \setminus \Omega$ and that
	\begin{equation}
		\int_{\Omega} c u \varphi \d x = 0.
	\end{equation}
	Hence, from the weak formulation we have
	\begin{equation}
		0 \geq \int_{\R^n \setminus \Omega} u \ncal_\s \varphi \d x \geq \int_{\Omega} u \fraclaplacian \varphi \d x + \int_{\R^n \setminus \Omega} u \ncal_\s \varphi \d x + \int_{\Omega} c u \varphi \d x  \geq 0
	\end{equation}
	and thus $u\equiv 0$ a.e. in $\R^n\setminus \Omega$. 
	The same argument yields
	\begin{equation}
		0 \geq \int_{\Omega\setminus B} u \fraclaplacian \varphi \d x = \int_{\Omega} u \fraclaplacian \varphi \d x + \int_{\R^n \setminus \Omega} u \ncal_\s \varphi \d x + \int_{\Omega} c u \varphi \d x  \geq 0
	\end{equation}
	and since $u\geq 0$ a.e. in $\R^n$ and $\fraclaplacian \varphi < 0$ in $\R^n \setminus B$, it follows that  $u\equiv 0$ a.e. in $\Omega\setminus B$, concluding the proof. 
\end{proof}

Now we recall a strong maximum principle for $L^1$-weak supersolutions in $H^\s_\Omega$.
It is used several times along the paper.

\begin{proposition}[Strong maximum principle]
	\label{Prop:StrongMaxPple}
	
	Let $\Omega\subset \R^n$ be a smooth bounded domain, and let $u\in H^\s_{\Omega}$ be a function such that $u\geq 0$ a.e. in $\R^n$ and satisfying
	\begin{equation}
		\label{Eq:SuperSHarmonicS}
		\fraclaplacian u + c u \geq 0 \quad  \text{ in } \Omega
	\end{equation} 
	in the $L^1$-weak sense and for some $c\in L^\infty(\Omega)$.
	
	If there exists a ball $B\subset \subset \Omega$ such that $\inf_B u = 0$, then $u\equiv 0$ a.e. in $\R^n$.
	
\end{proposition}

\begin{proof}
	First, proceeding as in \Cref{Remark:L1WeakSolWhichAreHs} it follows that $u$ is a weak supersolution to $\fraclaplacian u = f(u)$  in $\Omega$ in the sense of \cite[Definition~8.4]{Cozzi-DeGiorgiClassesLong}, and for $f$ such that $|f(u)| \leq d_1 + d_2 |u|^{q-1}$, with $q=2$, $d_1 = 0$, and $d_2 = \norm{c}_{L^\infty(\Omega)}$.
	Taking this into account, a detailed inspection\footnote{More precisely, following the notation of \cite[Proposition~8.5]{Cozzi-DeGiorgiClassesLong} one sees that if $p=2$ (which is the situation of the present paper), then one can take $k_0 = -\infty$.} 
	of the proof of Proposition~8.5 of \cite{Cozzi-DeGiorgiClassesLong} shows that $u$ fulfills the hypothesis of \cite[Proposition~6.8]{Cozzi-DeGiorgiClassesLong} with $d = d_1 = 0$, and therefore the following weak Harnack inequality holds:
	\begin{equation}
		\left( \average_B |u|^q \d x \right)^{1/q} \leq C \inf_B u 
	\end{equation}
	for some $q \in (0,1)$ and some constant $C$ depending on $n$, $\s$, $B$, and $\norm{c}_{L^\infty(\Omega)}$.
	Note that in nonlocal Harnack inequalities a tail term must appear in the right-hand side of the inequality if one considers the hypothesis of $u$ being nonnegative only on a ball $B$  (see \cite{DiCastroKuusiPalatucci-Harnack}). 
	This tail term is essentially the integral in $\R^n\setminus B$ of $u_-$ times a power growth function, and thus it does not appear here because $u\geq 0$ a.e. in $\R^n$.
	%Note that the reason why no tail term appears in the weak Harnack inequality is because $u\geq 0$ a.e. in $\R^n$.
	Since $\inf_B u = 0$ we get that $u\equiv 0$ a.e. in $B$, and the result follows from \Cref{Prop:UniqueContinuation}.
\end{proof}

We now show how to easily extend the $L^p$ estimates of \cite{RosOtonSerra-Extremal} to $L^1$-weak solutions and with nonzero exterior conditions.
The result that we use repeatedly through the paper (specially in some bootstrap arguments) is the following.

\begin{proposition}[$L^p$ regularity for $L^1$-weak solutions]
	\label{Prop:RegularityLp}
	Let $\s\in (0,1)$ and $n> 2\s$.
	Let $\Omega\subset \R^n$ be a smooth bounded domain, let $h \in L^1(\Omega, d_\Omega^\s \dx)$, and let $g$ satisfy \eqref{Eq:Assumptionsg} for some $\alpha_0 \in (\max\{0, 2\s - 1\},  \s)$ and some positive constant $C_0$.
	Assume that $u$ is the unique $L^1$-weak solution $u$ to \eqref{Eq:DirichletPbLinear} given by \Cref{Prop:ExistenceL1}.
	
	\begin{enumerate}[label=(\roman*)]
		\item If $h \in L^1(\Omega)$, then $u\in L^q(\Omega)$ for every $1\leq q < \frac{n}{n -2\s}$ and
		\begin{equation}
			\label{Eq:L1ToLp}
			\norm{u}_{L^q(\Omega)} \leq C \left(  \norm{h}_{L^1(\Omega)} + C_0 \right), \quad \text{ for } 1 \leq q <\frac{n}{n -2\s},
		\end{equation}
		for some constant $C$ depending only on $n$, $\s$, $q$, $\alpha_0$, and $\Omega$.
		
		\item If $h \in L^p(\Omega)$ for $p \in (1, \frac{n}{2\s})$, then $u\in L^{\frac{np}{n - 2\s p}}(\Omega)$ and
		\begin{equation}
			\label{Eq:LpToLq}
			\norm{u}_{L^{\frac{np}{n - 2\s p}}(\Omega)} \leq C \left ( \norm{h}_{L^p(\Omega)}  + C_0 \right), 
		\end{equation}
		for some constant $C$ depending only on $n$, $\s$,  $p$, $\alpha_0$, and $\Omega$.
		
		\item
		If $h \in L^p(\Omega)$ for $p \in (\frac{n}{2\s}, +\infty)$, then $u\in C^\beta(\overline{\Omega})$ and
		\begin{equation}
			\label{Eq:LpToLinfty}
			\norm{u}_{C^\beta(\overline{\Omega})} \leq C \left( \norm{h}_{L^p(\Omega)} + C_0 \right), \quad \text{ with } \beta = \min\left \{\alpha_0, 2\s - \frac{n}{p}\right \},
		\end{equation}
		for some constant $C$ depending only on $n$, $\s$, $p$, $\alpha_0$, and $\Omega$.
		
	\end{enumerate}
\end{proposition}

\begin{proof}
	We prove first  $(i)$ and $(ii)$ in the case $g\equiv 0$.
	As in the proof of \Cref{Prop:ExistenceL1}, by splitting $h= h^+- h^-$ we can assume without loss of generality that $h\geq0$, and thus $u\geq 0$ by the maximum principle; see \Cref{Prop:MaxPple}.
	Then, we consider the function
	\begin{equation}
		v := (-\Delta)^{-\s} (h \chi_{\Omega}),
	\end{equation}
	in the sense that $v$ is the Riesz potential of order $2\s$ of $h \chi_{\Omega}$.
	Since $v = (h \chi_{\Omega}) * |x|^{-n + 2\s}$ up to a normalizing constant, using that the fractional Laplacian and the convolution commute it is not difficult to see that 
	\begin{equation}
		\beqc{\PDEsystem}
		\fraclaplacian v & = & h& \text{ in } \Omega,\\
		v & \geq & 0 & \text{ in } \R^n\setminus \Omega,
		\eeqc
	\end{equation}
	in the $L^1$-weak sense, an therefore by the maximum principe it follows that $0 \leq u \leq v$.
	From this, $(i)$ and $(ii)$ (with $C_0 = 0$) follow readily from the classical embeddings for the Riesz potential exactly as in \cite[Proposition~1.4]{RosOtonSerra-Extremal}.
	
	To treat the case $g\not \equiv 0$, it suffices to write $u = \tilde{v} + w$, where $\tilde{v}$ solves the same problem as $u$ but with zero exterior condition (and for which we have already established the desired estimates), and $w$ solves
	\begin{equation}
		\label{Eq:ExteriorProblem}
		\beqc{\PDEsystem}
		\fraclaplacian w & = & 0& \text{ in } \Omega,\\
		w & = & g & \text{ in } \R^n\setminus \Omega,
		\eeqc
	\end{equation}
	and for which it holds that $\norm{w}_{ L^\infty (\Omega)}\leq C C_0$; see \cite[Lemma~2.1]{AudritoRosOton}.

	To establish $(iii)$, we take as before \begin{equation}
		v := (-\Delta)^{-\s} (h \chi_{\Omega}),
	\end{equation} 
	and we write $u = v + w$, with $w$ being the solution to
	\begin{equation}
		\label{Eq:ExteriorProblemLp}
		\beqc{\PDEsystem}
		\fraclaplacian w & = & 0& \text{ in } \Omega,\\
		w & = & g - v & \text{ in } \R^n\setminus \Omega.
		\eeqc
	\end{equation}
	From the embedding of the Riesz potential into Hölder spaces when $p > \frac{n}{2\s}$ (see \cite[Theorem~1.6]{RosOtonSerra-Extremal}), it follows that 
	\begin{equation}
		[v]_{C^{2\s - n/p}(\R^n)} \leq C \norm{h}_{L^p (\Omega)}
	\end{equation}
	for some constant $C$ depending only on $n$, $\s$, and $p$.
	Moreover, since $\Omega$ is bounded $h \chi_{\Omega}$ has compact support, and therefore $v$ has a decay at infinity (and in particular is bounded).
	Therefore, the exterior condition $g-v$ satisfies \eqref{Eq:Assumptionsg} but with an exponent $\beta = \min\{\alpha_0, 2\s - \frac{n}{p}\}$ instead of $\alpha_0$.
	From the boundary regularity results of \cite{AudritoRosOton}, it follows that
	\begin{equation}
		\label{Eq:EstimatesHarmonicExterior}
		\norm{w}_{C^{\beta}(\overline{\Omega})} \leq C C_0.
	\end{equation}

	Note that if $p > n$ the solution to \eqref{Eq:ExteriorProblemLp} is obtained by standard methods minimizing the $H^\s_\Omega$ seminorm in the convex set $ \{ \varphi \in H^\s_\Omega \ : \ \varphi - g + v \in H^\s_{\Omega,0}\}$.
	If $p\leq n$ (which can only happen if $s > 1/2$), then the regularity up to $\partial \Omega$ of the exterior condition $g- v$ may not be enough to define energy solutions (see the comments in \Cref{Sec:ExteriorCondition}), but we can consider $w$ being a viscosity solution\footnote{Note that viscosity solutions to \eqref{Eq:ExteriorProblem} are $L^1$-weak solutions. 
		Indeed, the existence and boundary regularity of viscosity solutions is proved in \cite{AudritoRosOton}. From this and the fact that $w \in C^\infty(\Omega)$ (since $\fraclaplacian w =0$) it follows that the equation holds pointwise and multiplying by a test function with compact support in $\Omega$ we can integrate by parts to obtain the $L^1$-weak formulation (note that the test functions with compact support are dense among the test functions considered in the $L^1$-weak formulation).
	} to \eqref{Eq:ExteriorProblemLp} instead, for which the estimate \eqref{Eq:EstimatesHarmonicExterior} holds as well.
\end{proof}

\begin{remark}
	\label{Remark:n<2s}
	The previous result does not provide any estimate for $n\leq 2s$ (which, since $\s\in (0,1)$, entails $n=1$ and $\s \leq 1/2$).
	Nevertheless, since in this case the Green function of any bounded domain (which is an interval) is explicit, we easily have an $L^\infty$ estimate provided that $h\in L^p(\Omega)$ for some $p>1$, as mentioned in Remark~1.5 of \cite{RosOtonSerra-Extremal}.
	This is used implicitly in all the bootstrap arguments carried out in this paper when the nonlinearity $f(\cdot)$ is Lipschitz and $n\leq 2s$.
	Indeed, since we consider solutions in $L^2(\Omega)\subset H^\s_\Omega$, the previous $L^\infty$ estimates yield the boundedness of the solution in the case $n\leq 2s$.
\end{remark}

To conclude the section, we present the method of sub and supersolutions (also called \textit{monotone iteration method}) in the $L^1$-weak setting.
We show how to find an $L^1$-weak solution to the semilinear problem 
\eqref{Eq:SemilinearProblem}
once we have a sub and a supersolution in the $L^1$-weak sense.

\begin{proposition}[Monotone iteration method in the $L^1$-setting]
	\label{Prop:MonotoneIteration}
	Let $\Omega \subset \R^n$ be a bounded smooth domain, and let $f \in C^1(\R)$ be a nondecreasing function.
	Assume that there exist two functions $\usup, \usub \in L^1_\s(\R^n)\cap L^1_\loc(\R^n\setminus \Omega, d^{-\s}_\Omega \d x)$ which are, respectively, a supersolution and a subsolution of \eqref{Eq:SemilinearProblem} in the $L^1$-weak sense, for some exterior condition $g$ as in \eqref{Eq:Assumptionsg}.
	Furthermore, assume that $\usub \leq \usup$ a.e. in $\Omega$.
	
	Then, there exists an $L^1$-weak solution $u$ to \eqref{Eq:SemilinearProblem} which is minimal relative to $\usub$, that is, $\usub \leq u \leq w$ for every $L^1$-weak supersolution $w$  such that $\usub \leq w$.
	In particular, $\usub \leq u \leq \usup$ a.e. in $\R^n$.
\end{proposition}

\begin{proof}
	Set $u_0:= \usub$ and, for $k\geq 1$, define recursively $u_k$ as the $L^1$-weak solution to
	\begin{equation}
		\label{Eq:MonotoneIterationPb}
		\beqc{\PDEsystem}
		\fraclaplacian u_k & = & f(u_{k-1}) & \text{ in } \Omega,\\
		u_k & = & g & \text{ in } \R^n\setminus \Omega.
		\eeqc
	\end{equation}
	Below we will see that in each step we have $f(u_{k-1}) \in L^1(\Omega, d_\Omega^\s\d x)$ and thus $u_k$ are well defined and uniquely determined by \Cref{Prop:ExistenceL1}.
	
	We claim that for $k\geq 0$, $\usub \leq u_k \leq u_{k+1} \leq \usup$ a.e. in $\Omega$.
	To prove it we proceed by induction, using the maximum principle (see \Cref{Prop:MaxPple}) and the monotonicity of $f$.
	First, note that $\fraclaplacian u_0 = \fraclaplacian \usub \leq f(\usub) = f(u_0) = \fraclaplacian u_1 = f(\usub) \leq f(\usup) \leq \fraclaplacian \usup$ in the $L^1$-weak sense in $\Omega$.
	Then, the maximum principle yields $\usub = u_0 \leq u_1 \leq  \usup$ a.e. in $\Omega$.
	Now, assume that we have $ u_{k-1} \leq u_{k} \leq \usup$ a.e. in $\Omega$.
	Then, using this hypothesis and the monotonicity of $f$, we have
	$ \fraclaplacian u_k =  f(u_{k-1}) \leq  f(u_k) = \fraclaplacian u_{k+1} = f(u_k) \leq f(\usup) \leq \fraclaplacian \usup$  in the $L^1$-weak sense in $\Omega$,
	and as before, the maximum principle yields $ u_{k} \leq u_{k+1} \leq \usup$ a.e. in $\Omega$, establishing our claim.
	Note that the previous argument, together with the monotonicity of $f$ and the fact that $f(\usub),f(\usup) \in L^1(\Omega, d_\Omega^\s\d x)$, yields that in each iteration step, after defining $u_k$ we get that $f(u_k) \in  L^1(\Omega, d_\Omega^\s\d x)$, guaranteeing the existence of $u_{k+1}$.
	
	Once the previous claim is proved, since the sequence $\{u_k\}_{k \geq 0}$ is nondecreasing and lies between the two integrable functions $\usub$ and $\usup$, we can define its pointwise limit $u$, and by dominated convergence $u_k \to u$ in $L^1(\Omega)$ as $k\to +\infty$.
	By continuity, we also have $f(u_{k} )\to f(u)$ a.e. in $\Omega$, and thus since $f(\usub) \leq f(u_{k}) \leq f(\usup)$, by the dominated convergence theorem, $f(u_{k} )\to f(u)$ in $L^1(\Omega, d_\Omega^\s\d x)$ as $k \to +\infty$.
	Taking all this into account, we can take the limit $k \to +\infty$ in the $L^1$-weak formulation of \eqref{Eq:MonotoneIterationPb} to obtain that $u$ is an $L^1$-weak solution to \eqref{Eq:SemilinearProblem}.
	
	Note that we have only used $\usup$ to guarantee some integrability properties, but since the iteration has been done from $\usub$, in the previous arguments we can replace $\usup$ by any other supersolution $w$ such that $w \geq \usub$ a.e. in $\Omega$, proving the minimality of $u$ relative to $\usub$.
\end{proof}

\begin{remark}
	\label{Remark:MonotoneIteration}
	The same result holds without the assumption of $f$ being nondecreasing, if we assume that $f$ is uniformly Lipschitz in the interval $(\inf_\Omega \usub , \sup_\Omega \usup)$.
	Indeed, if we consider the nonlinearity $\tilde{f} (t) :=  f(t) + \ell t$, where $\ell\geq 0$ is the Lipschintz constant of $f$ in the above interval, the problem is equivalent to find a solution to  $\fraclaplacian u  + \ell u  =  \tilde{f}(u) $ in $\Omega$ with the same exterior condition.
	Then, we can repeat the previous proof using that $\tilde{f}$ is nondecreasing and that the maximum principle holds in this case since $\ell\geq 0$ (see \Cref{Prop:MaxPple}).	
\end{remark}

\addtocontents{toc}{\SkipTocEntry}
\section*{Acknowledgements}

The author wants to thank Alessandro Audrito, Juan-Carlos Felipe-Navarro, Antonio J. Fernández, and Matteo Cozzi for fruitful discussions on the topic of this paper.

%%%%%%%%%%%%%%%%%%%%%%%%%%%%%%%%%%%%%%%%%%%%%%%%%%%%%%%%%%%%%%%%%%%%%%%%%%%%
%%%%%%%%%%%%%%%%%%%%%%%%%%%%%%%%%%%%%%%%%%%%%%%%%%%%%%%%%%%%%%%%%%%%%%%%%%%%
\bibliographystyle{amsplain}
\bibliography{biblio}

\end{document}